\documentclass[12pt,reqno,english]{amsart}%
\usepackage{amsfonts,amsmath,latexsym,verbatim,amscd,mathrsfs,color,array}
\usepackage[colorlinks=true]{hyperref}
\usepackage{amsmath,amssymb,amsthm,amsfonts,graphicx,color}
\usepackage{amssymb}
\usepackage{pdfsync}
\usepackage{epstopdf}
\usepackage{cite}
\usepackage{graphicx}
\usepackage{amsmath}
\usepackage{amsfonts}%
\providecommand{\U}[1]{\protect\rule{.1in}{.1in}}

\numberwithin{equation}{section}

\newtheorem{theorem}{Theorem}[section]
\newtheorem{corollary}{Corollary}[section]
\newtheorem{lemma}{Lemma}[section]

\newtheorem{remark}{Remark}[section]
\newtheorem{definition}{Definition}[section]

\numberwithin{equation}{section}

\newcommand{\bbr}{\mathbb{R}}

\newcommand{\bbn}{\mathbb{N}}
\newcommand{\ve}{\varepsilon}
\newcommand{\bd}{\begin{definition}}
\newcommand{\ed}{\end{definition}}

\newcommand{\br}{\begin{remark}}
\newcommand{\er}{\end{remark}}

\newcommand{\be}{\begin{equation}}
\newcommand{\ee}{\end{equation}}

\newcommand{\bc}{\begin{corollary}}
\newcommand{\ec}{\end{corollary}}
\begin{document}

\title[Positive solutions of the Gross--Pitaevskii equation]{Positive solutions of the Gross--Pitaevskii equation for energy critical and supercritical nonlinearities}

\author[D. E.  Pelinovsky]{Dmitry E.Pelinovsky}
\address{ Department of Mathematics and Statistics, McMaster University, Hamilton,
	Ontario, Canada, L8S 4K1}
\email{dmpeli@math.mcmaster.ca}

\author[J. Wei]{Juncheng Wei}
\address{\noindent Department of Mathematics, University of British Columbia,
Vancouver, B.C., Canada, V6T 1Z2}
\email{jcwei@math.ubc.ca}

\author[Y. Wu]{Yuanze Wu}
\address{\noindent  School of Mathematics, China
University of Mining and Technology, Xuzhou, 221116, P.R. China }
\email{wuyz850306@cumt.edu.cn}

\begin{abstract}
We consider positive and spatially decaying solutions to the following Gross--Pitaevskii equation with a harmonic potential:
\begin{eqnarray*}
-\Delta u+|x|^2u=\omega u+|u|^{p-2}u\quad\text{in }\mathbb{R}^d,
\end{eqnarray*}
where $d\geq3$, $p>2$ and $\omega>0$.  For $p=\frac{2d}{d-2}$ (energy-critical case) there exists a ground state $u_\omega$  if and only if $\omega\in(\omega_*, d)$, where $\omega_* =1$ for $d=3$ and $\omega_*=0$ for $d\geq4$.  We give a precise description on asymptotic behaviors of $u_\omega$ as $\omega\to\omega_*$ up to the leading order term for different values of $d \geq 3$.  When $ p>\frac{2d}{d-2}$ (energy-supercritical case) there exists a singular solution $u_\infty$ for  some $\omega\in(0, d)$. We compute the Morse index of $u_\infty$ in the class of radial functions and show   that the Morse index of $u_\infty$ is infinite in the oscillatory case, is equal to $1$ or $2$ in the monotone case for $p$ not large enough and is equal to $1$ in the monotone case for $p$ sufficiently large.

\vspace{3mm} \noindent{\bf Keywords:} Gross--Pitaevskii equation; Critical and supercritical nonlinearity; Positive solutions; Asymptotic behavior; Morse index.

\vspace{3mm}\noindent {\bf AMS} Subject Classification 2010: 35B09; 35B33; 35B40; 35J20.%

\end{abstract}

\date{}
\maketitle

\section{Introduction}

In this paper, we consider positive and spatially decaying solutions to the following stationary Gross--Pitaevskii equation with a harmonic potential:
\begin{eqnarray}\label{eq0001}
-\Delta u+|x|^2u=\omega u+|u|^{p-2}u\quad\text{in }\mathbb{R}^d,
\end{eqnarray}
where $d\geq3$, $p>2$ and $\omega>0$.

\vskip0.2in

The stationary equation \eqref{eq0001} is a classical model to describe
the Bose--Einstein condensate with attractive inter-particle interactions under magnetic trap (cf. \cite{TW98}) if $d=1,2,3$ and $p=4$ (the cubic case) or $p = 6$ (the quintic case). In this context, $\psi(t,x) = e^{-i\omega t} u(x)$ is a standing wave solution of the time-dependent Gross--Pitaevskii equation
\begin{eqnarray*}
i \partial_t\psi=-\Delta\psi+|x|^2\psi-|\psi|^{p-2}\psi\quad\text{in }\mathbb{R}^d,
\end{eqnarray*}
where $\psi$ stands for the macroscopic wave function, $|x|^2$ is an isotropic  trapping potential that confines the Bose--Einstein condensate, and the nonlinear term corresponds to attractive inter-atomic interactions. Positive and spatially decaying solutions are called the bright solitons in the physics literature.

\vskip0.2in

Since the operator $-\Delta+|x|^2$ is compact in $L^2(\bbr^d)$, the energy-subcritical case $2<p<\frac{2d}{d-2}$ can be studied by classical variational methods or bifurcation methods (cf. \cite{Fuk,KW94,S11}).
On the other hand, energy-critical $p = \frac{2d}{d-2}$ and energy-supercritical
$p > \frac{2d}{d-2}$ cases were less investigated in the literature
(cf. \cite{S11,SK12,SKW13}, \cite{BFPS21,Ficek,PS22}, and \cite{Gustaf}).

\vskip0.2in

In the energy-critical case $p = \frac{2d}{d-2}$ with $d \geq 3$, based on the well-known Gidas-Ni-Nirenberg theorem (cf. \cite{GNN81}), the existence of positive and spatially decaying solutions of the stationary equation \eqref{eq0001} has been shown in \cite{S11,SK12,SW13} for $\omega\in(\omega_*, d)$, where
\begin{eqnarray}\label{eq9911}
\omega_*=\left\{\aligned&1, \quad d=3,\\
&0,\quad d\geq4.
\endaligned
\right.
\end{eqnarray}
However, besides the existence and nonexistence of solutions,
it is also interesting for critical elliptic equations to study the concentration phenomenon and the asymptotic behavior of solutions
for $\omega \to \omega_*$. These problems were initialed by Brez\'is {\em et al.} in \cite{B86,BN83,BP89} in the context of
the following Dirichlet problem
\begin{eqnarray}\label{eq1002}
\left\{ \begin{array}{cl} -\Delta u+ a(x)u =\omega u+|u|^{\frac{4}{d-2}}u & \quad\text{in }\Omega,\\
u(x)=0, &\quad\text{on }\partial\Omega,
\end{array} \right.
\end{eqnarray}
where $\Omega\subset\bbr^d$ ($d\geq3$) is a bounded domain with smooth boundary and $a(x)$ is a smooth weight (cf. \cite{AGGPV21,CPP21,dDM06,D02,E04,FKK20,FKK21,H91,HV01,I15,IV16,IV18,MP02,MP03,MP10,P22,R89,R90}). The concentration phenomenon of solutions of the Dirichlet problem \eqref{eq1002} depends on the geometry of the domain $\Omega$.  More precisely, solutions concentrate around the critical points of the Robin function of the domain $\Omega$.  To our best knowledge, the concentration phenomenon and the asymptotic behavior of positive and spatially decaying solutions of the stationary equation \eqref{eq0001} in the energy-critical case $p=\frac{2d}{d-2}$ have not been studied yet.  {\em The first purpose of this paper} is to give a precise description of the latter problems.

\vskip0.2in

We shall now introduce some notations and definitions to state our main results. Let $X \subset L^2(\mathbb{R}^d)$ be the form domain space for the operator $-\Delta + |x|^2$ equipped with the norm
\begin{eqnarray*}
\|u\|_X := \bigg(\int_{\bbr^d} (|\nabla u|^2 + |x|^2 |u|^2) dx\bigg)^{\frac12},
\end{eqnarray*}
We also introduce the energy space $\Sigma := X \cap L^{\frac{2d}{d-2}}(\mathbb{R}^d)$. For fixed $\omega \in (\omega_*,d)$, we define
\begin{eqnarray}\label{eq0002}
\mathcal{I}_{\omega} = \inf_{v \in \Sigma}\bigg\{ I_{\omega}(v) : \quad  \|v\|_{{L^{\frac{2d}{d-2}}(\bbr^d)}}=1\bigg\},
\quad I_{\omega}(v) := \|v\|_X^2 - \omega \|v\|_{L^2(\bbr^d)}^2.
\end{eqnarray}
By the method of Lagrange's multipliers and the scaling transformation, $u =(\mathcal{I}_{\omega})^{\frac{d-2}{4}} v$ is a nontrivial solution of the stationary equation \eqref{eq0001} if $v$ is a minimizer of the variational problem \eqref{eq0002}. Based on the above observations, we can introduce the following definition.

\begin{definition}
We say $u_{\omega}$ is a ground state of the stationary equation \eqref{eq0001}
if $v_{\omega} \in \Sigma$ is a minimizer of the variational problem (\ref{eq0002}) such that $I_{\omega}(v_{\omega}) = \mathcal{I}_{\omega}$ and $u_{\omega} := (\mathcal{I}_{\omega})^{\frac{d-2}{4}} v_{\omega}$.
\end{definition}

Let
\begin{eqnarray}\label{eq0010}
U_{\ve}(x)=\ve^{\frac{d-2}{2}}[d(d-2)]^{\frac{d-2}{4}}\bigg(\frac{1}{\ve^2+|x|^2}\bigg)^{\frac{d-2}{2}}, \quad \ve > 0
\end{eqnarray}
be a family of the algebraic solitons (also called the Aubin-Talanti bubbles \cite{Aubin,Talenti})
which satisfy the elliptic problem
\begin{eqnarray}\label{eq9901}
-\Delta u=u^{\frac{d+2}{d-2}}, \quad u\in D^{1,2}(\bbr^d),
\end{eqnarray}
where $ D^{1,2} (\bbr^d)$ denotes the space of closure of $C_0^\infty (\bbr^d)$ under the norm $ \|\nabla \cdot\|_{L^2 (\bbr^d)} $. 
%It is well-known that $u \in L^{\frac{2d}{d-2}} (\bbr^d)$ if $ u \in D^{1,2} (\bbr^d)$.

\vskip0.2in

For the sake of simplicity, we also denote $U_{\ve = 1}$ by $U$.  It is well known (cf. \cite{Aubin,Talenti}) that $U_\ve$ for every $\ve > 0$ attains the best constant of the Sobolev embedding
\begin{equation}
\label{Sob-emb}
\| u \|_{L^{\frac{2d}{d-2}}(\mathbb{R}^d)} \leq \mathcal{S}^{-\frac{1}{2}}  \| \nabla u \|_{L^2(\mathbb{R}^d)},
\end{equation}
where $\mathcal{S}$ is given by
\begin{eqnarray}\label{eqnewnew1003}
\mathcal{S} = \inf_{v \in D^{2,1}(\bbr^d)} \{\|\nabla v\|_{L^2(\bbr^d)}^2 : \quad \| v \|_{L^{\frac{2d}{d-2}}(\bbr^d)}=1\}.
\end{eqnarray}
By the scaling transformation, if $v$ is a minimizer of the variational problem (\ref{eqnewnew1003}), then $u := (\mathcal{S})^{\frac{d-2}{4}} v$ is a solution
of the elliptic problem (\ref{eq9901}) given by the family of algebraic solutions (\ref{eq0010}) up to spatial translations in $\bbr^d$.

\vskip0.2in

Since the operator $-\Delta+|x|^2-\omega_*$ is positive in $X$, we can define the unique solution of the following inhomogeneous equation
\begin{eqnarray}\label{eq9902}
-\Delta u+(|x|^2-\omega_*) u=U_{\ve}^{\frac{d+2}{d-2}}, \quad u\in X,
\end{eqnarray}
denoted by $PU_\ve$. Moreover, since $U_\ve>0$, by the positivity of the operator $-\Delta+|x|^2-\omega_*$ and the maximum principle, we know that $PU_{\ve}>0$ in $\bbr^d$.

\vskip0.2in

Let $G$ be the Green function of the positive operator $-\Delta+|x|^2-\omega_*$,
\begin{eqnarray}\label{eq9904}
\left\{ \begin{array}{cl} -\Delta G+(|x|^2-\omega_*)G=(d-2)|\mathbb{S}^{d-1}|\delta_0 &\quad\text{in }\bbr^d,\\
G(x)\to 0 &\quad\text{as }|x|\to+\infty,\end{array} \right.
\end{eqnarray}
where $\delta_0$ is the Dirac measure supported at $x=0$ and $|\mathbb{S}^{d-1}|$ is the Lebesgue measure of the unit sphere in $\bbr^d$.
This gives the unique normalization of the Green function such that
$G = |x|^{2-d} - H$, where $H$ is a regular part of $G$ satisfying the following equation:
\begin{eqnarray}\label{eq9905}
\left\{\begin{array}{cl} -\Delta H+(|x|^2-\omega_*)H = (|x|^2-\omega_*)|x|^{2-d} &\quad\text{in }\bbr^d,\\
H(x)\to 0 &\quad\text{as }|x|\to+\infty.\end{array}\right.
\end{eqnarray}
By uniqueness of solutions to the elliptic problems \eqref{eq9904} and \eqref{eq9905}, $G$ and $H$ are radially symmetric. Our main results in the energy-critical case $p=\frac{2d}{d-2}$ can be stated as follows.

\begin{theorem}\label{thm0001}
Let $d\geq3$, $p=\frac{2d}{d-2}$, and $u_\omega$ be the ground state solution of the stationary equation \eqref{eq0001} for $\omega\in(\omega_*, d)$, where $\omega_*$ is given by \eqref{eq9911}.  There exsts $\ve_{\omega} > 0$ such that
\begin{itemize}
\item $u_\omega=PU_{\ve_\omega}+\hat{u}_{\omega}$ for $3\leq d\leq6$
\item $u_\omega=U_{\ve_\omega}+\hat{u}_{\omega}$ $\;\;$ for $d\geq7$,
\end{itemize}
with $\ve_\omega\to0$ and $\| \hat{u}_\omega\|_X\to0$ as $\omega\to\omega_*$.  Moreover, $\mathcal{I}_{\omega} < \mathcal{S}$ for $\omega \in (\omega_*,d)$,
$\mathcal{I}_\omega \mapsto \mathcal{S}$ as $\omega \to \omega_*$,
and there exist positive constants $a_d$ and $b_d$ which only depend on the dimension $d$, such that the concentration rate $\ve_\omega$ and the ground state energy $\mathcal{I}_\omega$ satisfy
\begin{itemize}
\item for $d=3$
 \begin{eqnarray*}
 a_d=\lim_{\omega\to1}\frac{\ve_\omega}{(\omega-1) \| G \|^2_{L^2(\bbr^3)}}, \qquad
 b_d=\lim_{\omega\to1}\frac{\mathcal{S}-\mathcal{I}_{\omega}}{((\omega-1) \| G \|^2_{L^2(\bbr^3)})^{2}},
 \end{eqnarray*}
\item for $d=4$,
 \begin{eqnarray*}
 a_d=\lim_{\omega\to0}\frac{\omega|\log\ve_\omega|}{H(0)\|U\|_{L^{3}(\bbr^4)}^{3}}, \qquad  b_d=\lim_{\omega\to0}\frac{\omega|\log(\mathcal{S}-\mathcal{I}_{\omega})-\log(c_dH(0)\|U\|_{L^{3}(\bbr^4)}^{3})|}{H(0)\|U\|_{L^{3}(\bbr^4)}^{3}},
 \end{eqnarray*}
\item for $d=5$,
 \begin{eqnarray*}
 a_d=\lim_{\omega\to0}\frac{H(0)\|U\|_{L^{\frac{7}{3}}(\bbr^d)}^{\frac{7}{3}}\ve_\omega}{\|U\|_{L^2(\bbr^5)}^2\omega}, \qquad  b_d=\lim_{\omega\to0}\frac{(H(0)\|U\|_{L^{\frac{7}{3}}(\bbr^d)}^{\frac{7}{3}})^2(\mathcal{S}-\mathcal{I}_{\omega})}{\|U\|_{L^2(\bbr^5)}^6\omega^3},
 \end{eqnarray*}
\item for $d=6$,
 \begin{eqnarray*}
 a_d=\lim_{\omega\to0}\frac{|\log\omega|\ve_\omega^2}{\|U\|_{L^2(\bbr^6)}^2\omega}, \qquad  b_d=\lim_{\omega\to0}\frac{|\log\omega|(\mathcal{S}-\mathcal{I}_{\omega})}{\|U\|_{L^2(\bbr^6)}^4\omega^2},
 \end{eqnarray*}
\item for $d\geq7$,
 \begin{eqnarray*}
 \frac{1}{2}=\lim_{\omega\to0}\frac{\||x|U\|_{L^2(\bbr^d)}^2\ve_\omega^2}{\|U\|_{L^2(\bbr^d)}^2\omega}, \qquad  b_d=\lim_{\omega\to0}\frac{\||x|U\|_{L^2(\bbr^d)}^2(\mathcal{S}-\mathcal{I}_{\omega})}{\|U\|_{L^2(\bbr^d)}^4\omega^2}.
 \end{eqnarray*}
\end{itemize}
\end{theorem}

\begin{remark}
To our best knowledge, Theorem~\ref{thm0001} is the first result on the concentration phenomenon and the asymptotic behavior of solutions of the stationary Gross--Pitaevskii equation \eqref{eq0001} in the energy-critical case $p=\frac{2d}{d-2}$.  It is worth pointing out that a formal and brief calculation on the upper bounds of $\mathcal{I}_{\omega}$ is obtained in \cite[Section 5]{S11} to ensure the existence of minimizers of $\mathcal{I}_{\omega}$.  These upper bounds of $\mathcal{I}_{\omega}$ are calculated in a standard way by choosing the Aubin-Talanti bubbles as test functions of $\mathcal{I}_{\omega}$, as that in \cite{BN83}.  However, the main difficulty in proving Theorem~\ref{thm0001} is to obtain a good lower bound of $\mathcal{I}_{\omega}$ which will match the upper bound generated by the Aubin-Talanti bubbles up to the leading order terms.  To achieve this, we need to further employ the ideas in literature \cite{BP89,D02,E04,FKK20,FKK21,H91,HV01,R89,R90}, that is, splitting of $u_\omega$ into two parts in $X$ and estimating of these two parts precisely up to the leading order term.  We remark that, due to the growth of the harmonic potential at infinity and the unboundedness of $\bbr^d$, the regular part of the Green function of the operator $-\Delta+|x|^2-\omega_*$ is no longer bounded for all $d\geq3$.  Thus, we need to modify the arguments of the proofs in a nontrivial way to capture the leading order terms of $\ve_\omega$ and $\mathcal{I}_{\omega}$ for all $d\geq3$, which also makes the concentration phenomenon of positive solutions of the stationary equation \eqref{eq0001} to be more complicated than that of the Dirichlet problem \eqref{eq1002}.
\end{remark}

\begin{remark}
	One can use parameter $\varepsilon$ in the family of algebraic solitons (\ref{eq0010}) to parameterize the family of the ground states $(\omega,u_{\omega})$ of the stationary equation (\ref{eq0001}). It follows from Theorem \ref{thm0001} that the asymptotic behavior of the mapping $\varepsilon \mapsto \omega$ as $\ve \to 0$ depends on the dimension $d \geq 3$ and satisfies
	$$
	\omega - \omega_* \sim
	\left\{ \begin{array}{ll}
\varepsilon & \quad \mbox{\rm for} \; d = 3, \\
|\log \varepsilon|^{-1} & \quad \mbox{\rm for} \; d = 4, \\
\varepsilon & \quad \mbox{\rm for} \; d = 5, \\
\varepsilon^2 |\log \varepsilon| & \quad \mbox{\rm for} \; d = 6, \\
\varepsilon^2 & \quad \mbox{\rm for} \; d \geq 7. \\
	\end{array} \right.
	$$
\end{remark}

\vskip0.2in

In the energy-supercritical case $p>\frac{2d}{d-2}$, we will fix $p=4$ to simplify the computations similarly to what was adopted in \cite{BFPS21,PS22}.  The energy-supercritical case for $p = 4$ corresponds to $d\geq5$ and the stationary equation \eqref{eq0001} is reduced to
\begin{eqnarray}\label{eqnew0001}
-\Delta u+|x|^2u=\omega u+u^3,\quad\text{in }\bbr^d.
\end{eqnarray}
It has been proved in \cite{SKW13} (see also \cite{BFPS21} for a different proof), that there exists a singular radial solution $u_{\infty}$
of the stationary equation (\ref{eqnew0001}) for some $\omega_\infty\in(d-4, d)$ satisfying
\begin{equation}
\label{sing-beh}
u_{\infty}(x) = \frac{\sqrt{d-3}}{|x|} \left[ 1 + \mathcal{O}(|x|^2) \right] \quad \mbox{\rm as} \;\; |x| \to 0.
\end{equation}
Moreover, by \cite[Theorem~1.1]{BFPS21}, for every $b>0$, there exists a positive radial solution $u_b$ of the stationary equation (\ref{eqnew0001}) for some $\omega_b\in(d-4, d)$ satisfying $u_b(0) = b$. By \cite[Theorem~1.2]{SKW13}, it is known that $u_b\to u_\infty$ strongly in $\Sigma$ and $\omega_b\to\omega_\infty$ as $b\to+\infty$.  The precise asymptotic behavior of $\omega_b$ as $b\to+\infty$ is obtained in \cite[Theorem~1.3]{BFPS21} under some nondegeneracy assumptions.  By \cite[Theorem~1.3]{BFPS21}, $\omega_b$ is oscillatory around $\omega_\infty$ as $b\to+\infty$ for $5\leq d\leq12$ and $\omega_b$ converges to $\omega_\infty$
monotonically as $b\to+\infty$ for $d\geq13$. These results suggest that the Morse index of $u_\infty$ in the class of radial functions is infinite for $5\leq d\leq12$ and finite for $d\geq13$, where the Morse index of $u_{\infty}$ is defined as follows.

\begin{definition}
	\label{def-Morse}
	Let $u_{\infty}$ be the singular radial solution of the stationary equation (\ref{eqnew0001}) for some $\omega_{\infty} \in (d-4,d)$ satisfying (\ref{sing-beh}) and consider the linearized operator
	$$
	L_{\infty} := - \Delta + |x|^2 - \omega_{\infty} - 3 u_{\infty}^2
	$$
	in $X_{\rm rad} := \{ f \in X : \text{$f$ is radial}\}$.	The Morse index of $u_{\infty}$ denoted by $\mathfrak{m}(u_{\infty})$ is the number of negative eigenvalues of $L_{\infty}$ in $X_{\rm rad}$.
\end{definition}

It was proven in \cite{PS22} that the Morse index of $u_b$ for large $b$ coincides with the Morse index of $u_{\infty}$ for $d \geq 13$. It was conjectured in \cite{PS22} based on numerical evidences that $\mathfrak{m}(u_{\infty}) = 1$ for $d\geq13$. {\em The second purpose of this paper} is to estimate $\mathfrak{m}(u_{\infty})$ which is done in the following main result.

\begin{theorem}\label{thm0002}
Let $p=4$, $d\geq5$ and $u_\infty$ be the singular radial solution of the stationary equation \eqref{eqnew0001} for some $\omega_\infty\in(d-4, d)$ satisfying (\ref{sing-beh}).  Then
\begin{eqnarray*}
\mathfrak{m}(u_\infty) = \left\{ \begin{array}{ll} \infty, & \quad 5\leq d\leq12,\\
1\text{ or }2, &\quad 13\leq d\leq 15,\\
1, & \quad d\geq16.
\end{array} \right.
\end{eqnarray*}
\end{theorem}

\begin{remark}
\item[$(i)$] To prove Theorem~\ref{thm0002} for $5\leq d\leq12$, we shall mainly follow the ideas in \cite{GW11}.  The oscillation of $\omega_b$ around $\omega_\infty$ as $b\to+\infty$ is obtained in \cite[Theorem~1.3]{BFPS21} under some nondegeneracy assumptions, which is hard to verify. In order to avoid making these nondegeneracy assumptions, we need to modify the arguments in \cite{GW11}.

\item[$(ii)$] In proving Theorem~\ref{thm0002} for $d\geq13$, we consider the limiting spectral problem
\begin{eqnarray}\label{eqnew2222}
-\Delta u+|x|^2u-\frac{3(d-3)}{|x|^2}u=\sigma u,\quad u\in X_{\rm rad},
\end{eqnarray}
whose eigenvalues $\{ \sigma_n \}_{n \in \mathbb{N}}$ are completely known in the literature from the confluent hypergeometric equation \cite{V16}. We compare $\omega_\infty+3u_\infty^2$ and $\sigma_3+\frac{3(d-3)}{r^2}$ to control $\mathfrak{m}(u_\infty)$ by the Morse index of the radial eigenfunctions of the spectral problem \eqref{eqnew2222}.  As a by-product, we also prove that $u_\infty$ is nondegenerate for $d\geq16$, this avoids the nondegeneracy assumptions of \cite{PS22}. See Remark~\ref{rmk0001} for more details.
\end{remark}

\vskip0.2in

\noindent{\bf\large Notations.} Throughout this paper, $C$ and $C'$ are indiscriminately used to denote various positive constants.
Notation $a\lesssim b$ means that there exists $C > 0$ such that $a\leq Cb$.
Notation $a = \mathcal{O}(b)$ means that there exist $C, C' > 0$ such that $C'b\leq a\leq Cb$. Notation $a = o(b)$ means that $\lim\limits_{b \to 0} a/b = 0$. Notation $a \sim b$ as $b \to 0$ means that $\lim\limits_{b \to 0} a/b = 1$ (the same convention is used if $b \to \infty$).

\section{The energy-critical case}

\subsection{Preliminaries}

It has been proved in \cite[Section~5]{S11}, without the statement of theorems, that $\mathcal{I}_{\omega}$ is attained for $\omega\in(\omega_*, d)$.  On the other hand, by the Pohozaev identity, see, e.g., \cite[Proposition 2.2]{BFPS21},  we know that the stationary equation \eqref{eq0001} has no solutions in $\Sigma$ for $\omega\leq0$ which implies that $\mathcal{I}_{\omega}$ can not be attained for $\omega\leq0$.  Moreover, since $d$ is the first eigenvalue of $-\Delta+|x|^2$ in $X$, by multiplying \eqref{eq0001} with the first eigenfunction of the operator $-\Delta+|x|^2$ on both sides and integrating by parts, see, e.g., \cite[Proposition 2.1]{BFPS21}, we know that the stationary equation \eqref{eq0001} has no positive solutions for $\omega\geq d$. This implies that $\mathcal{I}_{\omega}$ can not be attained for $\omega\geq d$ either since minimizers of the variational problem \eqref{eq0002} are positive and radially symmetric. In addition, by \cite[Theorem~3]{SK12} or \cite[Theorem~7]{SW13},  the stationary equation \eqref{eq0001} also has no positive solutions for $\omega\leq1$ in the case of $d=3$. Thus, we know that $\mathcal{I}_{\omega}$ is attained if and only if $\omega\in(\omega_*, d)$.

\vskip0.2in

Since $\mathcal{I}_{\omega}$ is attained for $\omega\in(\omega_*, d)$, it can be proven by a standard way that $\mathcal{I}_{\omega}$ is strictly decreasing for $\omega\in[\omega_*, d]$ with $\mathcal{I}_{\omega = \omega_*} = \mathcal{S}$ and $\mathcal{I}_{\omega = d} =0$, where $\mathcal{S}$ is the best constant of the Sobolev embedding given by the variational problem  \eqref{eqnewnew1003}.  The monotone property was first pointed out by Brez\'is and Nirenberg in \cite[Remark~1.5]{BN83}.  The detailed proofs were recently given in 
\cite[Lemma~2.1]{CZ12} and \cite[Lemma~3.3]{WW21}.  Hence,
we have
\begin{eqnarray}\label{eq0099}
0 < \mathcal{I}_{\omega} < \mathcal{S} = \mathcal{I}_{\omega_*} \quad\text{for all }\omega\in(\omega_*, d).
\end{eqnarray}

\vskip0.2in

Let $v_\omega$ be the minimizer of the variational problem (\ref{eq0002}) for $\omega\in(\omega_*, d)$. Then, $u_\omega := (\mathcal{I}_{\omega})^{\frac{d-2}{4}} v_\omega$ is the ground state solution of the stationary  equation \eqref{eq0001}. Since we are interested in $\omega\to\omega_*$ with $\omega_*<d$, it is standard to show that $\{u_\omega\}$ is bounded in $X$.  By the compactness of the embedding from $X$ to $L^2(\bbr^d)$ due to the harmonic potential $|x|^2$, we may assume that there exists $u_* \in \Sigma$ such that $u_\omega\rightharpoonup u_*$ weakly in $\Sigma$ and $u_\omega\to u_*$ strongly in $L^2(\bbr^d)$ as $\omega\to\omega_*$.
We claim that $u_* = 0$. Indeed, if $u_*\not=0$, then  $u_*\in \Sigma$ satisfies
\begin{equation}
\label{eq-u*}
-\Delta u_*+|x|^2u_*=\omega_* u_*+|u_*|^{\frac{4}{d-2}}u_*
\end{equation}
in the weak sense, which, together with \eqref{eq0099}, implies
\begin{eqnarray}\label{eq0098}
I_{\omega_*}(u_*) = \|u_*\|_{L^{\frac{2d}{d-2}}(\bbr^d)}^{\frac{4}{d-2}}\leq\|u_\omega\|_{L^{\frac{2d}{d-2}}(\bbr^d)}^{\frac{4}{d-2}}+o(1) = I_{\omega}(u_{\omega}) +o(1)\leq I_{\omega_*}(u_*)+o(1).
\end{eqnarray}
Thus, $u_*$ corresponds to the minimizer $v_*$ with $I_{\omega_*}(v_*) = \mathcal{I}_{\omega_*}$ by $u_* := (\mathcal{I}_{\omega_*})^{\frac{d-2}{4}} v_*$ so that $u_*$ is positive and radially symmetric. This contradicts the previously reviewed results, from which no positive and radially symmetric
solution of the stationary equation (\ref{eq-u*}) exists in $\Sigma$
with $\omega_*$ given by (\ref{eq9911}).
Therefore, we must have $u_*=0$ and $u_\omega\rightharpoonup0$ weakly in $X$ and $u_\omega\to 0$ strongly in $L^2(\bbr^d)$ as $\omega\to\omega_*$.  Moreover, since $v_\omega$ be the minimizer of the variational problem (\ref{eq0002}) and $u_\omega= (\mathcal{I}_{\omega})^{\frac{d-2}{4}} v_\omega$, by \eqref{eq0099} and \eqref{eq0098},
\begin{eqnarray}
\label{eq-new}
\|u_{\omega}\|_{L^{\frac{2d}{d-2}}(\bbr^d)}^{\frac{2d}{d-2}} =(\mathcal{I}_{\omega})^{\frac{d}{2}}\|v_{\omega}\|_{L^{\frac{2d}{d-2}}(\bbr^d)}^{\frac{2d}{d-2}}
=\mathcal{S}^{\frac{d}{2}}+o(1)\quad\text{as}\quad\omega\to\omega_*.
\end{eqnarray}
Since $u_\omega$ is also the ground state solution of the stationary  equation \eqref{eq0001}, by multiplying \eqref{eq0001} with $u_\omega$ on both sides and integrating by parts, we also have
\begin{eqnarray}\label{eqnewnew9875}
\|u_{\omega}\|_X^2=\omega\|u_{\omega}\|_{L^2(\bbr^d)}^2+\|u_{\omega}\|_{L^{\frac{2d}{d-2}}(\bbr^d)}^{\frac{2d}{d-2}}=\mathcal{S}^{\frac{d}{2}}+o(1)\quad\text{as}\quad\omega\to\omega_*
\end{eqnarray}
since $u_\omega\to 0$ strongly in $L^2(\bbr^d)$ as $\omega\to\omega_*$.  On the other hand, it follows from ~\eqref{Sob-emb}, \eqref{eq-new}, and \eqref{eqnewnew9875} that
\begin{eqnarray*}
\|\nabla u_{\omega}\|_{L^2(\bbr^d)}^2 = \mathcal{S}\|u_{\omega}\|_{L^{\frac{2d}{d-2}}(\bbr^d)}^{2}=\mathcal{S}^{1+\frac{d-2}{2}}+o(1)
=\mathcal{S}^{\frac{d}{2}}+o(1)\quad\text{as}\quad\omega\to\omega_*,
\end{eqnarray*}
which implies that
\begin{eqnarray}\label{eq0096}
\| x u_{\omega}\|^2_{L^2(\bbr^d)} = o(1) \quad\text{as}\quad\omega\to\omega_*.
\end{eqnarray}

\subsection{Expansions of $u_\omega$}

Since $u_\omega$ is a ground state of the stationary equation~\eqref{eq0001} related to a minimizer of the variational problem~\eqref{eq0002}, the moving-plane method (cf. \cite{GNN81}) or the Schwarz symmetrization (cf. \cite{W10}) imply that $u_{\omega}$ is radial, positive and strictly decreasing in $r=|x|$. The following lemma clarifies the construction of $PU_{\ve}$ from solutions of the inhomogeneous equation (\ref{eq9902}).

\begin{lemma}\label{lem0001}
Let $3\leq d\leq6$, then
\begin{eqnarray}\label{eq9918}
PU_\ve =  U_\ve - \ve^{\frac{d-2}{2}}[d(d-2)]^{\frac{d-2}{4}}H - \eta_\ve, \quad |x|\lesssim1
\end{eqnarray}
and
\begin{eqnarray}\label{eq9910}
PU_\ve(x) \lesssim \ve^{\frac{d+2}{2}}|x|^{-(4+d)}
\quad \text{ for } |x| \gtrsim1,
\end{eqnarray}
where $H$ is defined by (\ref{eq9905}) and the correction term
$\eta_{\ve}$ satisfies
\begin{eqnarray}\label{eq9999}
\|\eta_\ve\|_{L^{\infty}(\bbr^d)}\lesssim\ve^{\frac{d+2}{2}} \quad
\mbox{\rm for} \; 3 \leq d \leq 5,
\end{eqnarray}
and
\begin{eqnarray}\label{eq9917}
\|\eta_\ve\|_{W^{2,\frac{3}{2}}(\bbr^6)}\lesssim\ve^{4}, \quad
\mbox{\rm for} \; d = 6.
\end{eqnarray}
\end{lemma}

\begin{proof}
Since it follows from (\ref{eq0010})  that
\begin{eqnarray}\label{eq9909}
U_\ve(x) \sim \ve^{\frac{d-2}{2}} |x|^{2-d} \quad \text{ for } |x| \gtrsim1,
\end{eqnarray}
the classical $L^p$-theory of elliptic equations and the Sobolev embedding theorem imply that the unique solution of the inhomogeneous equation (\ref{eq9902}) exists and satisfies $PU_\ve\in L^{\infty}_{loc}(\bbr^d\backslash\{0\})$. In particular, $PU_\ve \lesssim 1$ for $|x| \gtrsim 1$ and $\ve\lesssim1$.  Since
\begin{eqnarray*}
-\Delta |x|^{-(4+d)} + (|x|^2-\omega_*) |x|^{-(4+d)} \sim |x|^{-(2+d)} \quad\text{for }|x|\gtrsim1,
\end{eqnarray*}
it follows from \eqref{eq9909} that $\ve^{\frac{d+2}{2}}|x|^{-(4+d)}$ is a supersolution of equation \eqref{eq9902} for $|x|\gtrsim1$.  Now, by the fact that $PU_\ve \lesssim1$ for $|x| \gtrsim 1$ and $\ve\lesssim1$, the fact that $PU_\ve \to 0$ and $\ve^{\frac{d+2}{2}}|x|^{-(4+d)}\to0$ as $|x|\to+\infty$ and the maximum principle, we obtain (\ref{eq9910}).

\vskip0.12in

To obtain (\ref{eq9918}), we write
\begin{eqnarray}\label{eqnew8840}
\varphi_\ve := U_\ve-PU_\ve,
\end{eqnarray}
then by \eqref{eq9901} and \eqref{eq9902}, $\varphi_\ve$ is the unique solution of the following equation:
\begin{eqnarray}\label{eq9903}
-\Delta u+(|x|^2-\omega_*)u=(|x|^2-\omega_*)U_\ve, \quad u\in X.
\end{eqnarray}
By \eqref{eq9911} and the maximum principle, $\varphi_\ve>0$ in $\bbr^d$ for $d\geq4$.  For $d=3$, since $PU_\ve>0$ in $\bbr^3$,
there exists a unique $r_0>0$ such that $\varphi_\ve$ is strictly increasing
with respect to $r = |x|$ in $[0, r_0)$ and is strictly decreasing in $[r_0, +\infty)$.    Moreover, it follows from \eqref{eq9904} by using the maximum principle that
\begin{eqnarray}\label{eq9906}
|G(x)| \lesssim e^{-\sigma |x|^2} \quad \text{ for some }\sigma>0,
\end{eqnarray}
so that $H(x) = |x|^{2-d} + \mathcal{O}(e^{-\sigma|x|^2})$ as $|x| \to \infty$.
Thus, by \eqref{eq9911} and the classical $L^p$-theory of elliptic equations, we know that $H\in W_{loc}^{2,s}(\bbr^d)$ for $1<s<3$ in the case of $d=3$, $1<s<+\infty$ in the case of $d=4$ and $1<s<\frac{d}{d-4}$ in the case of $d\geq5$.  It follows from the Sobolev embedding theorem that $H\in L^\infty(\bbr^d)\cap C_{loc}^{\alpha}(\bbr^d)$ for $3\leq d\leq5$ with $\forall\alpha\in(0,1)$ and $H\in L_{loc}^{\frac{3s}{3-s}}(\bbr^6)$ for $1<s<3$.
Next we define
\begin{eqnarray}\label{eq9912}
\eta_\ve := \varphi_\ve - \ve^{\frac{d-2}{2}}[d(d-2)]^{\frac{d-2}{4}}H.
\end{eqnarray}
It follows from \eqref{eq9905} and \eqref{eq9903} that $\eta_\ve$
is the unique solution of the following equation:
\begin{eqnarray*}
\left\{
\begin{array}{ll} -\Delta u+(|x|^2-\omega_*)u=\ve^{\frac{d-2}{2}}[d(d-2)]^{\frac{d-2}{4}}(|x|^2-\omega_*)g_\ve & \quad\text{in }\bbr^d,\\
u(x)\to0 &\quad\text{as }|x|\to+\infty,\end{array}\right.
\end{eqnarray*}
where $g_\ve=(\ve^2+|x|^2)^{\frac{2-d}{2}}-|x|^{2-d}$ satisfies
\begin{eqnarray}\label{eq9914}
g_\ve(x) \sim\left\{
\begin{array}{ll} -|x|^{2-d}, &\quad |x|\leq\frac{\ve}{\sqrt{2}},\\
-\ve^2 |x|^{-d},&\quad |x|\geq\frac{\ve}{\sqrt{2}}.\end{array} \right.
\end{eqnarray}
As in the previous estimates, by the classical $L^p$-theory of elliptic equations, the Sobolev embedding theorem and the maximum principle,
we obtain
\begin{eqnarray}\label{eq9915}
|\eta_\ve(x)|\lesssim\ve^{\frac{d+2}{2}} |x|^{-(2+d)}\quad\text{for } |x|\gtrsim1.
\end{eqnarray}
Let $h_\ve=\ve^{\frac{d-2}{2}}[d(d-2)]^{\frac{d-2}{4}}(|x|^2-\omega_*)g_\ve$. It follows from \eqref{eq9914} that
\begin{eqnarray}\label{eqnew9910}
\|h_\ve\|_{L^s_{loc}(\bbr^d)} \lesssim \left\{ \begin{array}{ll} \ve^{\frac{3}{s}-\frac{1}{2}}, & \quad d=3,\\
\ve^{2+\frac{d}{s}-\frac{d-2}{2}}, &\quad 4\leq d\leq6.
\end{array}\right.
\end{eqnarray}
Thus, by \eqref{eq9915} and the classical $L^p$-theory of elliptic equations, we know that $\eta_\ve\in W^{2,s}(\bbr^d)$ for $1<s<3$ in the case of $d=3$, $1< s <+\infty$ in the case of $d=4$ and $1< s <\frac{d}{d-4}$ in the case of $d\geq5$.
The Sobolev embedding theorem implies that $\eta_\ve\in L^\infty(\bbr^d)\cap C_{loc}^{\alpha}(\bbr^d)$ for $3\leq d\leq5$ with $\forall\alpha\in(0,1)$ and $\eta_\ve\in L^{\frac{3s}{3-s}}(\bbr^6)$ for $1<s<3$.  Representation \eqref{eq9918} follows from \eqref{eqnew8840} and \eqref{eq9912}.
Estimates \eqref{eq9999} and \eqref{eq9917} follow from \eqref{eq9915}, \eqref{eqnew9910}, the classical $L^p$-theory and the Sobolev embedding theorem by choosing $s=2$ for $d=3$ and $s=\frac{d}{d-2}$ for $d=4,5$.
\end{proof}

\vskip0.2in

By \eqref{eq0096} and Lions' theorem (cf. \cite[Theorem~1.41]{W96}), there exists $\{\ve_\omega\}\subset\bbr_+$ such that $u_\omega\to U_{\ve_\omega}$ strongly in $D^{1,2}(\bbr^d)$ as $\omega\to\omega_*$.  Since $u_\omega\to0$ strongly in $L^2_{loc}(\bbr^d)$ as $\omega\to\omega_*$, it is easy to see that $\ve_\omega\to0$ as $\omega\to\omega_*$. The following lemma specifies
a precise decomposition of $u_{\omega}$ near $U_{\ve_{\omega}}$.

\begin{lemma}\label{lem0002}
As $\omega\to\omega_*$, there exists $\ve_{\omega} > 0$ such that
\begin{eqnarray}\label{eq9920}
u_\omega = \left\{ \begin{array}{ll} PU_{\ve_\omega} + \hat{u}_{\omega} &\quad\text{for }3\leq d\leq6, \\
U_{\ve_\omega} + \hat{u}_{\omega} &\quad\text{for } d \geq 7, \end{array} \right.
\end{eqnarray}
where $\ve_{\omega} \to 0$ and $\hat{u}_{\omega} \to 0$ in $X$
as $\omega \to\omega_*$ and where $\hat{u}_{\omega} \in \mathcal{M}_{\omega}^{\perp}$ defined by
\begin{eqnarray*}
\mathcal{M}_{\omega} = \left\{ \begin{array}{ll} \{PU_{\ve_\omega}, \partial_{\ve_\omega}PU_{\ve_\omega}, \partial_{x_1}PU_{\ve_\omega},\cdots, \partial_{x_d}PU_{\ve_\omega}\} &\quad\text{for }3\leq d\leq6, \\
\{U_{\ve_\omega}, \partial_{\ve_\omega}U_{\ve_\omega}, \partial_{x_1}U_{\ve_\omega},\cdots, \partial_{x_d}U_{\ve_\omega}\}&\quad\text{for } d \geq 7, \end{array} \right.
\end{eqnarray*}
and the orthogonality holds simultaneously in $X$ and $L^2(\bbr^d)$.
\end{lemma}

\begin{proof}
It follows from the explicit formula \eqref{eq0010} for $d\geq7$ that
\begin{eqnarray}\label{eq9928}
\int_{\bbr^d}|x|^2U_{\ve}^2dx=\ve^4\int_{\bbr^d}|x|^2U^2dx,\quad\int_{\bbr^d}U_\ve^2dx=\ve^2\int_{\bbr^d}U^2dx.
\end{eqnarray}
Moreover, for all $d\geq3$,
\begin{eqnarray}
\int_{\bbr^d}U_\ve^qdx=\ve^{d-\frac{(d-2)q}{2}}\int_{\bbr^d}U^qdx\quad\text{for }q>\frac{d}{d-2}\label{eq9930}
\end{eqnarray}
and
\begin{eqnarray}
\int_{B_1}U_\ve^{\frac{d}{d-2}}dx\sim\ve^{\frac{d}{2}}|\log\ve|.\label{eq9957}
\end{eqnarray}
Thus, by the fact that $u_\omega\to U_{\ve_\omega}$ strongly in $D^{1,2}(\bbr^d)$ as $\omega\to\omega_*$ and \eqref{eq0096}, we have
\begin{eqnarray}\label{eq9929}
\|u_\omega-U_{\ve_\omega}\|_X^2\to0\quad\text{as }\omega\to\omega_*.
\end{eqnarray}
On the other hand, since $H,\eta_{\ve_\omega}\in L^\infty(\bbr^d)\cap C_{loc}^{\alpha}(\bbr^d)$ for $3\leq d\leq5$ with $\forall\alpha\in(0,1)$ and $H,\eta_{\ve_\omega}\in L^{\frac{3s}{3-s}}(\bbr^6)$ for $1<s<3$ by Lemma~\ref{lem0001}, it follows from \eqref{eq9918}, \eqref{eq9910}, \eqref{eq9999}, and \eqref{eq9917} that $PU_{\ve_\omega}\to U_{\ve_\omega}$ strongly in $D^{1,2}(\bbr^d)$ as $\omega\to\omega_*$.
Thus, it is also easy to see that
\begin{eqnarray}\label{eq9919}
\|u_\omega-PU_{\ve_\omega}\|_X^2\to0\quad\text{as }\omega\to\omega_*.
\end{eqnarray}
Now, we define
\begin{eqnarray*}
e(\omega) := \left\{ \begin{array}{ll} \inf\limits_{\ve\in\bbr_+}\|u_\omega-PU_{\ve}\|_X^2 &\quad\text{for }3\leq d\leq6, \\
\inf\limits_{\ve\in\bbr_+}\|u_\omega-U_{\ve}\|_X^2 &\quad\text{for } d\geq7.
\end{array} \right.
\end{eqnarray*}
By \eqref{eq9929} and \eqref{eq9919}, it is standard (cf. \cite{BC88,FKK21,R90}) to show that $e(\omega)=o_\omega(1)$ is attained by some $\ve_{\omega}$ satisfying $\ve_{\omega} \to0$ as $\omega\to\omega_*$, which implies that \eqref{eq9920}  hold with $\hat{u}_{\omega} \to0$ in $X$ as $\omega\to\omega_*$.  The orthogonality conditions in $X$ for
$\hat{u}_{\omega} \in \mathcal{M}_{\omega}^{\perp}$ are obtained from
\begin{eqnarray*}
\frac{d}{d\ve} \|u_\omega-PU_{\ve}\|_X^2 |_{\ve=\ve_\omega}=0\quad\text{for }3\leq d\leq6, \qquad
\frac{d}{d\ve} \|u_\omega-U_{\ve}\|_X^2 |_{\ve=\ve_\omega}=0 \quad\text{for } d\geq7.
\end{eqnarray*}
The orthogonality conditions in $L^2(\bbr^d)$ follows from the fact that the eigenfunctions of $-\Delta+|x|^2$ is a orthogonal basis of $L^2(\bbr^d)$.
\end{proof}

\subsection{Estimates on $\hat{u}_{\omega}$}

By \cite[Appendix~D]{R90},
\begin{eqnarray}\label{eq9821}
\int_{\bbr^d}
\left( |\nabla v|^2 -(2^*-1) U_{\ve_\omega}^{2^*-2} |v|^2 \right) dx \geq \frac{4}{d+4}\int_{\bbr^d}|v|^2dx
\end{eqnarray}
for all $v\in D^{1,2}(\bbr^d)$ satisfying
\begin{eqnarray*}
\int_{\bbr^d}\nabla v\nabla U_{\ve_\omega}dx=\int_{\bbr^d}\nabla v\nabla \partial_{\ve_\omega}U_{\ve_\omega}dx=\int_{\bbr^d}\nabla v\nabla \partial_{x_l}U_{\ve_\omega}dx=0
\end{eqnarray*}
where $l=1,2,\cdots,d$.  By Lemma~\ref{lem0002}, we have
\begin{eqnarray*}
\int_{\bbr^d}\nabla \hat{u}_{\omega}\nabla U_{\ve_\omega}dx=\int_{\bbr^d}\nabla \hat{u}_{\omega}\nabla \partial_{\ve_\omega}U_{\ve_\omega}dx=\int_{\bbr^d}\nabla \hat{u}_{\omega}\nabla \partial_{x_l}U_{\ve_\omega}dx=o(1)
\end{eqnarray*}
for all $l=1,2,\cdots,d$ as $\omega\to\omega_*$.  Thus,
\begin{eqnarray}\label{eq8821}
\int_{\bbr^d} \left( |\nabla \hat{u}_{\omega}|^2 -(2^*-1) U_{\ve_\omega}^{2^*-2}|\hat{u}_{\omega}|^2 \right) dx \geq
\left( \frac{4}{d+4}+o(1) \right) \int_{\bbr^d} |\hat{u}_{\omega}|^2 dx
\end{eqnarray}
for $d\geq7$.  On the other hand, by \eqref{eq9911}, \eqref{eq9918}--\eqref{eq9917}, \eqref{eq9821} and $\hat{u}_{\omega}\in \mathcal{M}_{\omega}^{\perp}$, it is also standard (cf. \cite[Appendix~D]{R90}) to show that
\begin{eqnarray}\label{eq9921}
\int_{\bbr^d}
\left( |\nabla \hat{u}_{\omega}|^2+(|x|^2-\omega_*)|\hat{u}_{\omega}|^2 -(2^*-1)PU_{\ve_\omega}^{2^*-2} |\hat{u}_{\omega}|^2 \right) dx \gtrsim \int_{\bbr^d}|\hat{u}_{\omega}|^2dx
\end{eqnarray}
for $3\leq d\leq6$. For $d=3$, we need to use the fact that $\omega_*=1$ and $\lambda = 3$ is the first eigenvalue of the operator $-\Delta+|x|^2$ in $L^2(\bbr^3)$.

The following lemma gives the asymptotic estimate
on the $X$ norm of $\hat{u}_{\omega}$. The proofs are simpler for $d \geq 7$ but get more technically involved for $3 \leq d \leq 6$.

\begin{lemma}\label{lem0003}
Let $d\geq3$, Then as $\omega\to\omega_*$,
\begin{eqnarray}\label{eq9935}
\|\hat{u}_{\omega} \|_X \lesssim \left\{ \begin{array}{ll} (\omega-1)\ve^{\frac{1}{2}}+\ve &\quad\text{for }d=3,\\
\omega\ve_\omega^{\frac{d}{2}-\sigma}+\ve_\omega^{\frac{d-2}{2}+\sigma} &\quad\text{for }4\leq d\leq6,\\
\omega\ve_\omega^{2} &\quad\text{for }d\geq7,\end{array}\right.
\end{eqnarray}
where $\sigma>0$ is a small constant.
\end{lemma}

\begin{proof}
	For $3 \leq d \leq 6$, we obtain from \eqref{eq0001}, \eqref{eq9902}, and \eqref{eq9920} that $\hat{u}_{\omega}$ satisfies
\begin{eqnarray}\label{eq9923}
\left\{\begin{array}{ll}
-\Delta \hat{u}_{\omega} + (|x|^2-\omega_*) \hat{u}_{\omega} -\frac{d+2}{d-2} PU_{\ve_\omega}^{\frac{4}{d-2}} \hat{u}_{\omega} = E_\omega + N_\omega(\hat{u}_{\omega}) & \quad\text{in }\bbr^d,\\
\hat{u}_{\omega}(x)\to0 &\quad\text{as }|x|\to0\end{array}\right.
\end{eqnarray}
where the nonhomogeneous term is
$$
E_\omega := (\omega-\omega_*)PU_{\ve_\omega}+PU_{\ve_\omega}^{\frac{d+2}{d-2}}-U_{\ve_\omega}^{\frac{d+2}{d-2}}
$$
and the nonlinear term satisfies
\begin{eqnarray}
\label{eq9925}
| N_\omega(\hat{u}_{\omega}) | \lesssim  PU_{\ve_\omega}^{\frac{6-d}{d-2}} |\hat{u}_{\omega}|^2+|\hat{u}_{\omega}|^{\frac{d+2}{d-2}}.
\end{eqnarray}
It follows from \eqref{eq9910} and \eqref{eq9909} that
\begin{eqnarray}\label{eq9926}
|E_\omega| \lesssim \ve_{\omega}^{\frac{d+2}{2}}((\omega-\omega_*) |x|^{-(4+d)} + |x|^{-(2+d)})\quad\text{for } |x| \gtrsim1.
\end{eqnarray}
For $|x| \lesssim1$, it follows from \eqref{eq9918} for $4\leq d\leq6$
(for which $\omega_*= 0$) that
\begin{eqnarray*}
|E_\omega|&\lesssim&\omega U_{\ve_{\omega}}+U_{\ve_{\omega}}^{\frac{4}{d-2}}(\ve_\omega^{\frac{d-2}{2}} |H| + |\eta_{\ve_{\omega}}|)
+U_{\ve_\omega}^{\frac{6-d}{d-2}}(\ve_\omega^{d-2}|H|^2+|\eta_{\ve_\omega}|^2)\\
&&+\ve_\omega^{\frac{d+2}{2}}|H|^{\frac{d+2}{d-2}}+|\eta_{\ve_\omega}|^{\frac{d+2}{d-2}},
\end{eqnarray*}
where $\varphi_{\ve_{\omega}} > 0$ is given by \eqref{eqnew8840} and \eqref{eq9912}. By Lemma~\ref{lem0001}, we obtain from \eqref{eq9930} for $4\leq d\leq5$ and $R>0$ sufficiently large,
\begin{align}
\left|\int_{B_{R}}E_\omega \hat{u}_{\omega} dx \right|
\lesssim & \omega\|\hat{u}_{\omega}\|_{L^{\frac{2d}{d-2}}(B_R)} \|U_{\ve_{\omega}}\|_{L^{\frac{d}{d-2}+\sigma}(B_R)}
+\int_{B_{R}}U_{\ve_{\omega}}^{\frac{4}{d-2}}(\ve_\omega^{\frac{d-2}{2}} |H|
+|\eta_{\ve_{\omega}}|) \hat{u}_{\omega} dx\notag\\
& +\int_{B_{R}}U_{\ve_\omega}^{\frac{6-d}{d-2}}(\ve_\omega^{d-2}|H|^2+|\eta_{\ve_\omega}|^2) \hat{u}_{\omega}  dx +\int_{B_{R}}(\ve_\omega^{\frac{d+2}{2}}|H|^{\frac{d+2}{d-2}}+|\eta_{\ve_\omega}|^{\frac{d+2}{d-2}}) \hat{u}_{\omega}  dx\notag\\
\lesssim & (\omega\|U_{\ve_{\omega}}\|_{L^{\frac{d}{d-2}+\sigma}(B_R)}+\ve_\omega^{\frac{d+2}{2}})\|\hat{u}_{\omega} \|_{L^{\frac{2d}{d-2}}(B_R)}\notag\\
& +\int_{B_{R}}U_{\ve_{\omega}}^{\frac{4}{d-2}}(\ve_\omega^{\frac{d-2}{2}} |H|
+ |\eta_{\ve_{\omega}}|) \hat{u}_{\omega} dx +\int_{B_{R}}U_{\ve_\omega}^{\frac{6-d}{d-2}}(\ve_\omega^{d-2}|H|^2+|\eta_{\ve_\omega}|^2) \hat{u}_{\omega} dx\notag\\
\lesssim & (\omega\ve_\omega^{\frac{d}{2}-\sigma}+\ve_\omega^{\frac{d+2}{2}})\|\hat{u}_{\omega}\|_{L^{\frac{2d}{d-2}}(\bbr^d)}+|I|,\label{eq9933}
\end{align}
where
\begin{eqnarray*}
I=\int_{B_{R}}U_{\ve_{\omega}}^{\frac{4}{d-2}}(\ve_\omega^{\frac{d-2}{2}} |H|
+ |\eta_{\ve_{\omega}}|) \hat{u}_{\omega} dx+\int_{B_{R}}U_{\ve_\omega}^{\frac{6-d}{d-2}}(\ve_\omega^{d-2}|H|^2+|\eta_{\ve_\omega}|^2) \hat{u}_{\omega} dx.
\end{eqnarray*}
Note that $H,\eta_{\ve_\omega}\in L^\infty(\bbr^d)$ for $4\leq d\leq5$ and $H,\eta_{\ve_\omega}\in L^{\frac{3s}{3-s}}(\bbr^d)$ for $d=6$ with $1<s<3$ by Lemma~\ref{lem0001}.  Thus, for $4\leq d\leq5$, it follows from \eqref{eq9999} and \eqref{eq9930} that
\begin{align*}
|I|&\lesssim \ve_\omega^{\frac{d-2}{2}}\|U_{\ve_\omega}\|_{L^{\frac{8d}{d^2-4}}(B_R)}^{\frac{4}{d-2}}\|\hat{u}_{\omega}\|_{L^{\frac{2d}{d-2}}(B_R)} +\ve_\omega^{d-2}\|U_{\ve_\omega}\|_{L^{\frac{2d(6-d)}{d^2-4}}(B_R)}^{\frac{6-d}{d-2}}\|\hat{u}_{\omega}\|_{L^{\frac{2d}{d-2}}(B_R)}\\
&\lesssim \ve_\omega^{\frac{d-2}{2}+\frac{4}{d+2}}\|\hat{u}_{\omega}\|_{L^{\frac{2d}{d-2}}(B_R)}\\
&\lesssim \ve_\omega^{\frac{d-2}{2}+\sigma}\|\hat{u}_{\omega}\|_{L^{\frac{2d}{d-2}}(\bbr^d)},
\end{align*}
and for $d=6$, it follows by \eqref{eq9917} that
\begin{align*}
|I| &\lesssim \ve_\omega^2\|U_{\ve_\omega}\|_{L^{\frac{p}{p-1}}(B_R)}\|\hat{u}_{\omega}\|_{L^{3}(B_R)}
+\ve_\omega^4\|\hat{u}_{\omega}\|_{L^{3}(B_R)}\\
&\lesssim \ve_\omega^{2+\sigma}\|\hat{u}_{\omega}\|_{L^{3}(\bbr^d)}.
\end{align*}
It follows from \eqref{eq9926} and \eqref{eq9933} that for $4\leq d\leq6$,
\begin{eqnarray}\label{eqnew7750}
\left| \int_{\bbr^d}E_\omega \hat{u}_{\omega} dx \right| \lesssim(\omega\ve_\omega^{\frac{d}{2}-\sigma}+\ve_\omega^{\frac{d-2}{2}+\sigma})\|\hat{u}_{\omega}\|_{L^{2^*}(\bbr^d)}.
\end{eqnarray}
For $d=3$, the estimates are similar to that of $d=4,5$.  The difference is that $\omega_* = 1$ and we do not know if $\varphi_{\ve_{\omega}}>0$ in $\bbr^3$.  Thus, we write
\begin{align*}
|E_\omega| \lesssim &(\omega-1)U_{\ve_{\omega}}+U_{\ve_{\omega}}^{\frac{4}{d-2}}(\ve_\omega^{\frac{d-2}{2}} |H| + |\eta_{\ve_{\omega}}|)
+ U_{\ve_\omega}^{\frac{6-d}{d-2}}(\ve_\omega^{d-2}|H|^2+|\eta_{\ve_\omega}|^2)\\
&+\ve_\omega^{\frac{d+2}{2}}|H|^{\frac{d+2}{d-2}}+|\eta_{\ve_\omega}|^{\frac{d+2}{d-2}}+(\omega-1)3^{\frac14}\ve_{\omega}^{\frac12}|H|+(\omega-1)|\eta_{\ve_{\omega}}|,
\end{align*}
which implies that
\begin{eqnarray}\label{eqnew7751}
\left| \int_{\bbr^d}E_\omega \hat{u}_{\omega} dx \right| \lesssim ((\omega-1)\ve_\omega^{\frac{1}{2}}+\ve_\omega)\|\hat{u}_{\omega}\|_{L^{2^*}(\bbr^d)}.
\end{eqnarray}
Estimates \eqref{eqnew7750} and \eqref{eqnew7751}, together with \eqref{eq9921}, \eqref{eq9923} and \eqref{eq9925}, imply \eqref{eq9935} for $3\leq d\leq6$.

For $d\geq7$, we obtain from \eqref{eq0001}, \eqref{eq9901}, and \eqref{eq9920} that $\hat{u}_{\omega}$ satisfies
\begin{eqnarray}\label{eq9823}
\left\{\begin{array}{ll} -\Delta \hat{u}_{\omega}+|x|^2 \hat{u}_{\omega} -\frac{d+2}{d-2}U_{\ve_\omega}^{\frac{4}{d-2}} \hat{u}_{\omega} = E_\omega + N_\omega(\hat{u}_{\omega}) &\quad\text{in }\bbr^d,\\
\hat{u}_{\omega}(x)\to0 &\quad\text{as }|x|\to0,\end{array}\right.
\end{eqnarray}
where $E_{\omega} := \omega U_{\ve_\omega}$ and
\begin{equation}
\label{eq9824}
|N_\omega(\hat{u}_{\omega})| \lesssim |\hat{u}_{\omega}|^{\frac{d+2}{d-2}}.
\end{equation}
It follows from \eqref{eq9909} that
\begin{align*}
\left|\int_{B_{R}}\widehat{E}_\omega \hat{u}_{\omega} dx \right| &\lesssim \omega\|U_{\ve_\omega}\|_{L^{2(\bbr^d)}}\| \hat{u}_{\omega} \|_{L^{2}(\bbr^d)}\notag\\
&\lesssim \omega\ve_\omega^{2}\|\hat{u}_{\omega}\|_{L^{2}(\bbr^d)},
\end{align*}
which, together with \eqref{eq8821}, \eqref{eq9823} and \eqref{eq9824} implies \eqref{eq9935} for $d\geq7$.
\end{proof}

\subsection{Asymptotic behaviors of $\mathcal{I}_{\omega}$ and $\ve_{\omega}$ as $\omega \to \omega_*$}

It follows from \eqref{eq0001} and (\ref{eq0002}) that if
$u_{\omega} = (\mathcal{I}_{\omega})^{\frac{d-2}{4}} v_{\omega}$, then
\begin{eqnarray}
\label{eq9950}
\mathcal{I}_{\omega} =\frac{\|u_\omega\|_X^2-\omega\|u_\omega\|_{L^2(\bbr^d)}^2}{\|u_\omega\|_{L^{\frac{2d}{d-2}}(\bbr^d)}^2} = \|u_\omega\|_{L^{\frac{2d}{d-2}}(\bbr^d)}^{\frac{4}{d-2}},
\end{eqnarray}
which yields
\begin{eqnarray}
\label{I-omega}
	\mathcal{I}_{\omega}
=(\|u_\omega\|_X^2-\omega\|u_\omega\|_{L^2(\bbr^d)}^2)^{\frac{2}{d}}.
\end{eqnarray}
The following four lemmas give details in the derivation of Theorem \ref{thm0001} for different values of $d \geq 3$.
The derivation is simpler for $d \geq 7$ and becomes computationally
challenging for $3 \leq d \leq 6$ due to different leading order terms in the expansion of $\mathcal{I}_{\omega}$ and due to different regularity of
the non-singular part $H$ of Green's function. Some similar computations
can be found in \cite{BN83,B86,FKK20,FKK21,R90} for $d \geq 4$ and
in \cite{D02,E04,FKK21,H91} for $d = 3$.

\begin{lemma}\label{lem0004}
For $d\geq7$, we have
\begin{eqnarray}\label{eq9943}
\mathcal{I}_{\omega}=\mathcal{S}-\mathcal{S}^{-\frac{d-2}{2}}\frac{\|U\|_{L^2(\bbr^d)}^4}{2d\|xU\|_{L^2(\bbr^d)}^2}\omega^2+o(\omega^2)
\end{eqnarray}
and
\begin{eqnarray}\label{eq9941}
\ve_\omega=\bigg(\frac{\|U\|_{L^2(\bbr^d)}^2}{2\|xU\|_{L^2(\bbr^d)}^2}\omega\bigg)^{\frac{1}{2}}+o(\omega^{\frac12})
\end{eqnarray}
as $\omega\to0$.
\end{lemma}

\begin{proof}
By \eqref{eq9920}, \eqref{eq9928}, and the estimates of Lemma \ref{lem0003}, we obtain from (\ref{I-omega}):
\begin{eqnarray}
\mathcal{I}_{\omega}&=&\bigg(\|U_{\ve_\omega}\|_X^2-\omega\ve_\omega^2\|U\|_{L^2(\bbr^d)}^2+o(\omega\ve_\omega^2)\bigg)^{\frac{2}{d}}\notag\\
&=&\bigg(\mathcal{S}^{\frac{d}{2}}+\ve_\omega^4\|xU_{\ve_\omega}\|_{L^2(\bbr^d)}^2
-\omega\ve_\omega^2\|U\|_{L^2(\bbr^d)}^2+o(\omega\ve_\omega^2)\bigg)^{\frac{2}{d}}\notag\\
&=&\mathcal{S}+\frac{2}{d}\mathcal{S}^{-\frac{d-2}{2}}(\ve_\omega^4\|xU\|_{L^2(\bbr^d)}^2
-\omega\ve_\omega^2\|U\|_{L^2(\bbr^d)}^2)+o(\omega\ve_\omega^2)\label{eqnew9840}.
\end{eqnarray}
On the other hand, by using $\{ U_{\ve} \}_{\ve > 0}$ as a test function of $\mathcal{I}_{\omega}$ for $d\geq7$, we obtain
\begin{align}
\label{eq9938}
\mathcal{I}_{\omega} &\leq \frac{\|U_{\ve}\|_X^2-\omega\|U_{\ve}\|_{L^2(\bbr^d)}^2}{\|U_{\ve}\|_{L^{2^*}(\bbr^d)}^2}\notag\\
&= \mathcal{S}+\frac{2}{d}\mathcal{S}^{-\frac{d-2}{2}}(\ve^4\|xU\|_{L^2(\bbr^d)}^2
-\omega \ve^2\|U\|_{L^2(\bbr^d)}^2).
\end{align}
Minimizing the right hand side of \eqref{eq9938} in terms of $\ve$ implies that
\begin{eqnarray}\label{eq9942}
\mathcal{I}_{\omega}&\leq&\mathcal{S}-\mathcal{S}^{-\frac{d-2}{2}}\frac{\|U\|_{L^2(\bbr^d)}^4}{2d\|xU\|_{L^2(\bbr^d)}^2}\omega^2.
\end{eqnarray}
Thus, combining \eqref{eqnew9840} and \eqref{eq9942}, we have  \eqref{eq9943} and \eqref{eq9941}.
\end{proof}

\vskip0.2in

\begin{lemma}\label{lem0005}
For $d=6$, we have
\begin{eqnarray}\label{eq9954}
\mathcal{I}_{\omega}=\mathcal{S}-\mathcal{S}^{-2}\frac{\|U\|_{L^2(\bbr^d)}^4\omega^2}{24^3|\mathbb{S}^5||\log\omega|}
+o\left(\frac{\omega^2}{\log\omega}\right)
\end{eqnarray}
and
\begin{eqnarray}\label{eq9955-double}
\ve_\omega=\bigg(\frac{\|U\|_{L^2(\bbr^d)}^2\omega}{12\times24^2|\mathbb{S}^5||\log\omega|}\bigg)^{\frac{1}{2}}+o\left(\frac{\omega}{|\log\omega|}\right)
\end{eqnarray}
as $\omega\to0$.
\end{lemma}

\begin{proof}
With $d = 6$, expression (\ref{I-omega}) becomes
\begin{eqnarray*}
\mathcal{I}_{\omega}=(\|u_\omega\|_X^2-\omega\|u_\omega\|_{L^2(\bbr^6)}^2)^{\frac{1}{3}}.
\end{eqnarray*}
By Lemmas~\ref{lem0002} and \ref{lem0003}, we have
\begin{eqnarray}\label{eqnew4450}
\|u_\omega\|_X^2=\|PU_{\ve_{\omega}}\|_X^2+ \mathcal{O}(\omega^2\ve_{\omega}^{6-\sigma}+\ve_{\omega}^{4+\sigma}),
\end{eqnarray}
where $\sigma>0$ is a small constant given by Lemma~\ref{lem0003} and if necessary, $\sigma$ can be taken arbitrary small.
By Lemma~\ref{lem0001}, we obtain from \eqref{eq9902} that
\begin{eqnarray}
\|PU_{\ve_{\omega}}\|_X^2&=&\int_{\bbr^6}U_{\ve_{\omega}}^2PU_{\ve_{\omega}}dx\notag\\
&=&\mathcal{S}^{3}-24\ve_{\omega}^{2}\int_{B_R}U_{\ve_{\omega}}^2Hdx+
\mathcal{O}(\ve_{\omega}^4),\label{eqnew4451}
\end{eqnarray}
where $R>0$ is sufficient large.
Since $H\in L^\infty_{loc}(\bbr^6\backslash\{0\})$, we need to further expand $H$ in $B_R$.  Since $\Delta(\log |x|)=\frac{4}{|x|^2}$ in $\bbr^6$ in the sense of distributions, it follows from \eqref{eq9905} that $\widehat{H} := H+\frac{1}{4}\log |x|$
satisfies the following equation:
\begin{eqnarray*}
-\Delta \widehat{H}+|x|^2\widehat{H}=|x|^2\log|x|\quad\text{in }\bbr^6,
\end{eqnarray*}
in the sense of distributions.  Since $|x|^2\log|x|\in W^{1,\infty}_{loc}(\bbr^6)$, by the classical elliptic regularity, $\widehat{H}\in C^{2}_{loc}(\bbr^6)$.  It follows that $H=-\frac{1}{4}\log |x|+ \mathcal{O}(1)$ in $B_R$ as $R \to \infty$.
Thus, we obtain from \eqref{eq9957} that
\begin{align*}
\int_{B_R}U_{\ve_{\omega}}^2Hdx&=-\frac{1}{4}\int_{B_R}U_{\ve_{\omega}}^2\log|x|dx+ \mathcal{O}\left(\int_{B_R}U_{\ve_{\omega}}^2dx \right)\\
&=144|\mathbb{S}^5|\ve_{\omega}^2|\log\ve_{\omega}|+
\mathcal{O}(\ve_{\omega}^{2+\sigma}),
\end{align*}
which, together with \eqref{eqnew4450} and \eqref{eqnew4451}, implies
\begin{eqnarray}\label{eqnew4452}
\|u_\omega\|_X^2=\mathcal{S}^{3}-6\times24^2|\mathbb{S}^5|\ve_{\omega}^4|\log\ve_{\omega}|+\mathcal{O}(\ve_{\omega}^4).
\end{eqnarray}
Similarly, by Lemmas~\ref{lem0001} and \ref{lem0003} and the expansion $H=-\frac{1}{4}\log |x| +O(1)$ in $B_R$ for any sufficiently large $R>0$, we have
\begin{eqnarray}\label{eqnew4453}
\int_{\bbr^6}PU_{\ve_{\omega}}^2dx=\ve_{\omega}^2\|U\|_{L^2(\bbr^d)}^2+o(\ve_{\omega}^2).
\end{eqnarray}
Thus, by \eqref{eqnew4452} and \eqref{eqnew4453}, we have
\begin{eqnarray}\label{eq9845}
\mathcal{I}_{\omega}=\mathcal{S}+\frac13\mathcal{S}^{-2}(6\times24^2|\mathbb{S}^5|\ve_{\omega}^4|\log\ve_{\omega}|-\omega\ve_{\omega}^2\|U\|_{L^2(\bbr^d)}^2
+o(|\ve_{\omega}^4|\log\ve_{\omega}|+\omega\ve_{\omega}^2)).
\end{eqnarray}

On the other hand, by using $W_\ve := (U_\ve-24\ve^2H)\phi_R$, where $\phi_R\in[0, 1]$ is a smooth cut-off function such that $\phi_R=1$ for $|x|\leq R$ and $\phi_R=0$ for $|x|\geq R+1$, as a test function of $\mathcal{I}_{\omega}$ for $d=6$, we have
\begin{eqnarray*}
\mathcal{I}_{\omega}&\leq&\frac{\|W_{\ve}\|_X^2-\omega\|W_{\ve}\|_{L^2(\bbr^d)}^2}{\|W_{\ve}\|_{L^{2^*}(\bbr^d)}^2},
\end{eqnarray*}
which implies
\begin{eqnarray}\label{eqnew9845}
\mathcal{I}_{\omega}\leq\mathcal{S}+\frac13\mathcal{S}^{-2}(6\times24^2|\mathbb{S}^5|\ve^4|\log\ve|-\omega\ve^2\|U\|_{L^2(\bbr^6)}^2
+o(|\ve^4|\log\ve|+\omega\ve^2)).
\end{eqnarray}
Minimizing the right hand side of \eqref{eqnew9845} in terms of $\ve$ implies that
\begin{eqnarray}\label{eq9952}
\mathcal{I}_{\omega}\leq\mathcal{S}-\mathcal{S}^{-2}\frac{\|U\|_{L^2(\bbr^d)}^4\omega^2}{12\times24^2|\mathbb{S}^5||\log\omega|}
+o(\frac{\omega^2}{|\log\omega}|).
\end{eqnarray}
Thus, by combining \eqref{eq9845} and \eqref{eq9952}, we have \eqref{eq9954} and \eqref{eq9955-double}.
\end{proof}

\begin{lemma}\label{lem0006}
For $d=4,5$, we have
\begin{eqnarray}\label{eq9956}
\mathcal{I}_{\omega}=\left\{\aligned&\mathcal{S}-\sqrt{2}\mathcal{S}^{-2}H(0)\|U\|_{L^{3}(\bbr^d)}^{3}e^{\frac{3\sqrt{2}H(0)\|U\|_{L^{3}(\bbr^d)}^{3}}{2\omega|\mathbb{S}^3|}}
+o(e^{-\frac{1}{\omega}}),\quad d=4,\\
&\mathcal{S}-\mathcal{S}^{-\frac{5}{2}}\frac{54\|U\|_{L^2(\bbr^5)}^6}{1715\times15^{\frac{3}{2}}(H(0)\|U\|_{L^{\frac{7}{3}}(\bbr^d)}^{\frac{7}{3}})^2}\omega^3+o(\omega^3),\quad d=5,\endaligned\right.
\end{eqnarray}
and
\begin{eqnarray}\label{eq9955}
\ve_{\omega}=\left\{\aligned&e^{-\frac{3\sqrt{2}H(0)\|U\|_{L^{3}(\bbr^d)}^{3}}{4\omega|\mathbb{S}^3|}}+o(e^{-\frac{1}{\omega}}),\quad d=4,\\
&\frac{3\|U\|_{L^2(\bbr^5)}^2}{7\times15^{\frac{3}{4}}H(0)\|U\|_{L^{\frac{7}{3}}(\bbr^d)}^{\frac{7}{3}}}\omega+o(\omega),\quad d=5\endaligned\right.
\end{eqnarray}
as $\omega\to0$.
\end{lemma}
\begin{proof}
	Different from the proofs of Lemmas~\ref{lem0004} and \ref{lem0005}, we use the relation (\ref{eq9950}). Moreover, we note that $H(0)$ is the maximum of $H(x)$ by the maximum principle, which implies from (\ref{eq9905}) for $d = 4,5$ that $H(0)>0$. Recall that $H,\eta_{\ve_\omega}\in L^\infty(\bbr^d)\cap C_{loc}^{\alpha}(\bbr^d)$ for $d = 4,5$ with $\forall\alpha\in(0,1)$ by Lemma~\ref{lem0001}.  Then by Lemma~\ref{lem0003}, we have
\begin{eqnarray*}
\|u_\omega\|_{L^{\frac{2d}{d-2}}(\bbr^d)}^{\frac{2d}{d-2}}=\|PU_{\ve_{\omega}}\|_{L^{\frac{2d}{d-2}}(\bbr^d)}^{\frac{2d}{d-2}}
+\frac{2d}{d-2}\int_{\bbr^d}PU_{\ve_{\omega}}^{\frac{d+2}{d-2}}\hat{u}_{\omega}dx
+\mathcal{O}(\omega^2\ve_\omega^{d-\sigma}+\ve_\omega^{d-2+\sigma}).
\end{eqnarray*}
By Lemma~\ref{lem0002}, it follows from \eqref{eq9902} that
\begin{eqnarray*}
\int_{\bbr^d}PU_{\ve_{\omega}}^{\frac{d+2}{d-2}}\hat{u}_{\omega}dx=\int_{\bbr^d}(PU_{\ve_{\omega}}^{\frac{d+2}{d-2}}-U_{\ve_\omega}^{\frac{d+2}{d-2}})\hat{u}_{\omega}dx.
\end{eqnarray*}
Then by similar calculations for \eqref{eq9933} and using \eqref{eq9918}, \eqref{eq9920} and \eqref{eq9935} and Lemma~\ref{lem0002} similar to that of $d=6$,
\begin{align}
\|u_\omega\|_{L^{\frac{2d}{d-2}}(\bbr^d)}^{\frac{2d}{d-2}}&=
\|PU_{\ve_{\omega}}\|_{L^{\frac{2d}{d-2}}(\bbr^d)}^{\frac{2d}{d-2}}
+\mathcal{O}(\omega^2\ve_\omega^{d-\sigma}+\ve_\omega^{d-2+\sigma})\notag\\
&=\mathcal{S}^{\frac{d}{2}}+\mathcal{O}(\ve_{\omega}^4)-\frac{2d}{d-2}\int_{B_R}U_{\ve_\omega}^{\frac{d+2}{d-2}}(\ve_{\omega}^{\frac{d-2}{2}}[d(d-2)]^{\frac{d-2}{4}}H+\eta_{\ve_\omega})dx\notag\\
&\quad +\mathcal{O}(\int_{B_R}U_{\ve_\omega}^{\frac{4}{d-2}}(\ve_{\omega}^{\frac{d-2}{2}}H+\eta_{\ve_\omega})^2+((\ve_{\omega}^{\frac{d-2}{2}} H+\eta_{\ve_\omega})^{\frac{2d}{d-2}})dx)\notag\\
&=\mathcal{S}^{\frac{d}{2}}-\frac{2d^{\frac{d+2}{4}}}{(d-2)^{\frac{6-d}{4}}}\ve_{\omega}^{\frac{d-2}{2}}\int_{B_R}U_{\ve_\omega}^{\frac{d+2}{d-2}}Hdx+\mathcal{O}(\ve_{\omega}^{d-2+\sigma}),\label{eq9948}
\end{align}
where we have used \eqref{eq9930} and $R>0$ is a large constant.  By the regularity of $H$ for $d=4,5$, $H(x)=H(0)+\mathcal{O}(|x|^{\alpha})$ near $|x|=0$.  Therefore, we can choose $\rho>0$ sufficiently small and obtain
\begin{align}
\int_{B_R}U_{\ve_\omega}^{\frac{d+2}{d-2}}Hdx &=H(0)\int_{B_{\rho}}U_{\ve_\omega}^{\frac{d+2}{d-2}}dx+\mathcal{O}(\int_{B_{\rho}}U_{\ve_{\omega}}^{\frac{d+2}{d-2}}|x|^{\alpha}dx)+\mathcal{O}(\int_{B_{R}\backslash B_{\rho}}U_{\ve_{\omega}}^{\frac{d+2}{d-2}}dx)\notag\\
&=\ve_\omega^{\frac{d-2}{2}}H(0)\|U\|_{L^{\frac{d+2}{d-2}}(\bbr^d)}^{\frac{d+2}{d-2}}+\mathcal{O}(\ve_\omega^{\frac{d-2}{2}+\alpha}).\label{eq9949}
\end{align}
It follows from \eqref{eq9950}, \eqref{eq9948} and \eqref{eq9949} that
\begin{eqnarray}\label{eq9951}
\mathcal{I}_{\omega}=\mathcal{S}-\frac{2}{d}\mathcal{S}^{-\frac{d-2}{2}}\ve_\omega^{d-2}[d(d-2)]^{\frac{d-2}{4}}H(0)\|U\|_{L^{\frac{d+2}{d-2}}(\bbr^d)}^{\frac{d+2}{d-2}}
+\mathcal{O}(\ve_\omega^{d-2+\sigma}),
\end{eqnarray}
where $\sigma>0$ is a small constant given by Lemma~\ref{lem0003} and if necessary, $\sigma$ can be taken arbitrary small.
On the other hand, by \eqref{eq0001}, we have
\begin{eqnarray}\label{eqnew4460}
\|u_\omega\|_X^2-\omega\|u_\omega\|_{L^2(\bbr^d)}^2=\|u_\omega\|_{L^{\frac{2d}{d-2}}(\bbr^d)}^{\frac{2d}{d-2}}.
\end{eqnarray}
Moreover, by using \eqref{eq9918}, \eqref{eq9999}, \eqref{eq9920} and \eqref{eq9935}, as in Lemma \ref{lem0005} we obtain
\begin{align*}
\|u_\omega\|_X^2 &= \|PU_{\ve_\omega}\|_X^2+\mathcal{O}(\omega^2\ve_\omega^{d-\sigma}+\ve_\omega^{d-2+\sigma})\\
&=\int_{\bbr^d}U_{\ve_{\omega}}^{\frac{d+2}{d-2}}PU_{\ve_{\omega}}dx+\mathcal{O}(\omega^2\ve_\omega^{d-\sigma}+\ve_\omega^{d-2+\sigma})\\
&=\mathcal{S}^{\frac{d}{2}}+\ve_\omega^{d-2}[d(d-2)]^{\frac{d-2}{4}}H(0)\|U\|_{L^{\frac{d+2}{d-2}}(\bbr^d)}^{\frac{d+2}{d-2}}+\mathcal{O}(\ve_\omega^{\frac{d-2}{2}+\alpha})
\end{align*}
and
\begin{align*}
\|u_\omega\|_{L^2(\bbr^d)}^2&=\|PU_{\ve_\omega}\|_{L^2(\bbr^d)}^2+2\int_{\bbr^d}PU_{\ve_\omega}\hat{u}_{\omega} dx+\mathcal{O}(\omega^2\ve_\omega^{d-\sigma}+\ve_\omega^{d-2+\sigma})\\
&=\left\{\begin{array}{ll} 8|\mathbb{S}^3|\ve_{\omega}^2|\log\ve_{\omega}|+o(|\ve_{\omega}^2|\log\ve_{\omega}|), &\quad d=4,\\
\ve_{\omega}^2\|U\|_{L^2(\bbr^5)}^2+o(\ve_{\omega}^2), &\quad d=5.\end{array}\right.
\end{align*}
By using \eqref{eqnew4460}, we obtain
\begin{align*} & \frac{d+2}{d-2}\ve_\omega^{d-2}[d(d-2)]^{\frac{d-2}{4}}H(0)\|U\|_{L^{\frac{d+2}{d-2}}(\bbr^d)}^{\frac{d+2}{d-2}}+o(\ve_\omega^{d-2}) \\
& = \left\{\begin{array}{ll} 8\omega|\mathbb{S}^3|\ve_{\omega}^2|\log\ve_{\omega}|+o(|\ve_{\omega}^2|\log\ve_{\omega}|), &\quad d=4,\\
\omega\ve_{\omega}^2\|U\|_{L^2(\bbr^5)}^2+o(\ve_{\omega}^2), &\quad d=5,\end{array}\right.
\end{align*}
which, together with \eqref{eq9951}, implies \eqref{eq9956} and \eqref{eq9955}.
\end{proof}

\begin{remark}
Two methods have been used to compute $\ve_\omega$.  The first one is to use the variational formula~\eqref{eq0002} which works for minimizers.  In this method, one need to use test functions in the variational formula to determine $\ve_\omega$.  This method is used for $d\geq6$ in Lemmas \ref{lem0004} and \ref{lem0005}.  The other one is to use the equation~\eqref{eq0001} which works for solutions (not necessary to be minimizers of variational problems) satisfying the decompositions in Lemmas~\ref{lem0002} and \ref{lem0003}.  In this method, we expand the equation~\eqref{eqnew4460} and determine $\ve_\omega$.  This method is used for $d=3,4,5$ in Lemmas \ref{lem0006} and \ref{lem0007}.
\end{remark}

\begin{lemma}\label{lem0007}
For $d=3$, we have
\begin{eqnarray}\label{eq9862}
\mathcal{I}_{\omega}=\mathcal{S}-\mathcal{S}^{-\frac{3}{2}}\frac{3^{\frac{3}{4}}\|G\|_{L^2(\bbr^3)}^4}{40\pi}(\omega-1)^2+o((\omega-1)^2)
\end{eqnarray}
and
\begin{eqnarray}\label{eq9861}
\ve_{\omega}=\frac{3^{\frac{5}{4}}\|G\|_{L^2(\bbr^3)}^2}{20\pi}(\omega-1)+o(\omega-1)
\end{eqnarray}
as $\omega\to1$.
\end{lemma}

\begin{proof}
Since the proof is similar to that of Lemma~\ref{lem0006}, we only sketch it.
We still use the relations (\ref{eq9950}) and (\ref{I-omega}) rewritten for $d = 3$ as
\begin{eqnarray}\label{eq19950}
\mathcal{I}_{\omega}=\frac{\|u_\omega\|_X^2-\omega\|u_\omega\|_{L^2(\bbr^3)}^2}{\|u_\omega\|_{L^{6}(\bbr^3)}^2}=\|u_\omega\|_{L^{6}(\bbr^3)}^{4}.
\end{eqnarray}
and
\begin{eqnarray}\label{eqnew14460}
\|u_\omega\|_X^2-\omega\|u_\omega\|_{L^2(\bbr^3)}^2=\|u_\omega\|_{L^{6}(\bbr^3)}^{6}.
\end{eqnarray}
As that in the proofs of Lemma~\ref{lem0005} and \ref{lem0006}, by \eqref{eq9935}, we have
\begin{equation*}
\|u_\omega\|_{L^2(\bbr^3)}^2 = \|PU_{\ve_\omega}\|_{L^2(\bbr^3)}^2+2\int_{\bbr^3}PU_{\ve_\omega}\hat{u}_{\omega} dx+\mathcal{O}((\omega-1)^2\ve_\omega^{3-\sigma}+\ve_\omega^{4+\sigma}).
\end{equation*}
By using \eqref{eq9918}, \eqref{eq9999}, \eqref{eq9906} and \eqref{eq9920}, we obtain
\begin{align*}
\|PU_{\ve_\omega}\|_{L^2(\bbr^3)}^2&=\int_{B_{\frac{1}{\ve_{\omega}}}}(U_{\ve_\omega}-\ve_{\omega}^{\frac12}3^{\frac{1}{4}}H-\eta_{\ve_{\omega}})^2+\mathcal{O}(\ve_{\omega}^{\frac{5}{2}})\\
&=\int_{B_{\frac{1}{\ve_{\omega}}}}(U_{\ve_\omega}-\ve_{\omega}^{\frac12}3^{\frac{1}{4}}H)^2dx+\mathcal{O}(\ve_{\omega}^{2-\sigma})\\
&=\ve_{\omega}3^{\frac{1}{2}}\int_{B_{\frac{1}{\ve_{\omega}}}}G^2dx-2\ve_{\omega}^{\frac12}3^{\frac{1}{4}}\int_{B_{\frac{1}{\ve_{\omega}}}}H(U_{\ve_{\omega}}-\ve_{\omega}^{\frac12}3^{\frac{1}{4}}|x|^{-1})dx\\
& \qquad +\int_{B_{\frac{1}{\ve_{\omega}}}}U_{\ve_{\omega}}^2-\ve_{\omega}3^{\frac{1}{2}}|x|^{-2}dx+\mathcal{O}(\ve_{\omega}^{2-\sigma})\\
&=\ve_{\omega}3^{\frac{1}{2}}\int_{\bbr^3}G^2dx+\mathcal{O}(\ve_{\omega}^{2-\sigma}).
\end{align*}
Moreover, similar to that of \eqref{eq9948}, by \eqref{eq9918}, \eqref{eq9999}, \eqref{eq9935} and Lemmas~\ref{lem0002} and \ref{lem0003},
\begin{align}
\|u_\omega\|_{L^{6}(\bbr^3)}^{6}&=\|PU_{\ve_{\omega}}\|_{L^{6}(\bbr^3)}^{6}+\mathcal{O}((\omega-1)^2\ve_{\omega}+\ve_\omega^{2})\notag\\
&=\mathcal{S}^{\frac{3}{2}}-\int_{B_R}6\sqrt[4]{3}\ve_{\omega}^{\frac12}U_{\ve_\omega}^{5}H-15\sqrt{3}\ve_{\omega}^2U_{\ve_\omega}^{4}H^2dx +\mathcal{O}((\omega-1)^2\ve_{\omega}+\ve_{\omega}^{2}).
\label{eq9848}
\end{align}
Since $H$ is not $C^1$, we need to expand $H(x)$ as that in \cite{FKK21}.  We define $\psi=H(x)-\frac{1}{2} |x|$, then by \eqref{eq9905}, $\psi$ satisfies
\begin{eqnarray*}
-\Delta\psi+(|x|^2-1)\psi=\frac{1}{2} |x| (1-|x|^2) \quad\text{in }\bbr^3.
\end{eqnarray*}
Since the data $\frac{|x|-|x|^3}{2}$ belongs to $W^{1,\infty}_{loc}(\bbr^3)$.  Thus, by the classical regularity theory, $\psi\in C^{2,\alpha}_{loc}(\bbr^3)$ for some $\alpha\in(0, 1)$, which, together with $\psi$ being radial, implies $\nabla\psi(0)=0$.  It follows that
\begin{eqnarray*}
H(x)=H(0)+\frac{1}{2}|x|+\mathcal{O}(|x|^2)\quad\text{near }|x|=0.
\end{eqnarray*}
Therefore, we obtain
\begin{eqnarray}\label{eqnew6668}
\int_{B_R}U_{\ve_\omega}^{5}Hdx=\frac{4\pi}{3}H(0)\ve_{\omega}^{\frac{1}{2}}-\frac{4\pi}{3}\ve_{\omega}^{\frac{3}{2}}+\mathcal{O}(\ve_{\omega}^{\frac{5}{2}}|\log\ve_{\omega}|)
\end{eqnarray}
and
\begin{eqnarray}\label{eqnew6667}
\int_{B_R}U_{\ve_\omega}^{4}H^2dx=H(0)^2\pi^2\ve_{\omega}+\mathcal{O}(\ve_{\omega}^2|\log\ve_{\omega}|).
\end{eqnarray}
Now, as that in the proof of Lemma~\ref{lem0006}, we obtain by using  \eqref{eqnew14460} and \eqref{eq9848},
\begin{align}
& \int_{B_R}5\sqrt[4]{3}\ve_{\omega}^{\frac12}U_{\ve_\omega}^{5}H-15\sqrt{3}\ve_{\omega}U_{\ve_\omega}^{4}H^2dx+o(\ve_{\omega}^2) \notag \\
&=
-(\omega-1)\ve_{\omega}3^{\frac{1}{2}}\int_{\bbr^3}G^2dx+o((\omega-1)\ve_{\omega}).\label{eq9860}
\end{align}
Dividing $\ve_{\omega}$ on both sides of \eqref{eq9860} and letting $\omega\to1$, we have $H(0)=0$ by \eqref{eqnew6668} and \eqref{eqnew6667}.  Thus, by \eqref{eqnew6668} and \eqref{eqnew6667} once more, \eqref{eq9860} is reduced to
\begin{eqnarray*}
\frac{20\sqrt[4]{3}\pi}{3}\ve_{\omega}^{2}+o(\ve_{\omega}^2)=(\omega-1)\ve_{\omega}3^{\frac{1}{2}}\int_{\bbr^3}G^2dx+o((\omega-1)\ve_{\omega}),
\end{eqnarray*}
which, together with \eqref{eq19950}, implies that  \eqref{eq9862} and \eqref{eq9861}.
\end{proof}

\begin{remark}
We note that $H(0)$ is a global minimum of $H(x)$ in $\bbr^3$. Indeed, by the maximum principle, it is easy to see that there exists $r_0\geq1$ such that $H(r)$ is strictly increasing in $[0, r_0]$ and is strictly decreasing in $[r_0, +\infty)$.  Thus, by $H(0)=0$ and $H(x)\to0$ as $|x|\to+\infty$, we have that $H(0)$ is actually a global minimum of $H(x)$.
\end{remark}

The proof of Theorem~\ref{thm0001} follows immediately from Lemmas~\ref{lem0004}--\ref{lem0007}.
\hfill$\Box$

\section{The energy-supercritical case}

\subsection{Preliminaries} 
Let $u_{\infty}$ be the singular solution of the stationary equation
(\ref{eqnew0001}) for some $\omega_\infty\in(d-4, d)$ satisfying
(\ref{sing-beh}) for $d \geq 5$. Let $L_{\infty}$ be the associated linear operator
given by
$$
L_{\infty} := - \Delta + |x|^2 - \omega_{\infty} - 3 u_{\infty}^2.
$$
Since $u_{\infty}(r) = \mathcal{O}(r^{-1})$ as $r \to 0$, $u_{\infty} \in C^{\infty}(0,\infty)$, and
$u_{\infty}(r) \to 0$ exponentially fast as $r \to +\infty$, we consider $L_{\infty}$ in the form domain $X_{\rm rad} := \{ f \in X : \text{$f$ is radial} \}$. The singular potential is controlled in the form domain by
using the following Hardy inequality for every $d \geq 3$:
\begin{eqnarray}
\label{Hardy}
\| |\cdot|^{-1} f \|_{L^2(\mathbb{R}^d)} \leq  \frac{2}{d-2} \| \nabla f \|_{L^2(\mathbb{R}^d)},\quad\forall f \in D^{1,2}(\bbr^d).
\end{eqnarray}
where $D^{1,2}(\bbr^d)$ is the same as in (\ref{eq9901}).

In order to justify the definition of Morse index $\mathfrak{m}(u_{\infty})$
according to Definition \ref{def-Morse}, we show that the
linear operator $L_{\infty}$ has a compact resolvent, which
implies that its spectrum of $L_{\infty}$ in $X_{\rm rad}$ is purely discrete
and consists of isolated (simple) eigenvalues.

\begin{lemma}
	\label{lem-compact}
	For every $d \geq 5$, the linear operator $L_{\infty}$ has a compact resolvent in $X_{\rm rad}$.
\end{lemma}

\begin{proof}
Consider the following variational problem:
	\begin{eqnarray*}
		\tau_1=\inf_{\phi\in X_{\rm rad}} \frac{\int_{\bbr^d}(|\nabla\phi|^2+(|x|^2-3u_\infty^2)|\phi|^2)dx}{\int_{\bbr^d}|\phi|^2dx}.
	\end{eqnarray*}
Since $F(r) := r u_{\infty}(r)$ is monotonically decreasing (cf. \cite{BFPS21,SKW13}), we have $F(r) < F(0) = \sqrt{d-3}$, which implies that $u_{\infty}(r) < \frac{\sqrt{d-3}}{r}$ for every $r > 0$.
By Hardy's inequality (\ref{Hardy}), we obtain
	\begin{eqnarray*}
		\int_{\bbr^d} 3u_\infty^2|\phi|^2dx \leq \int_{\bbr^d}\frac{3(d-3)}{|x|^2}|\phi|^2dx \leq
		\frac{12(d-3)}{(d-2)^2} \| \nabla\phi \|_{L^2(\mathbb{R}^d)}.
	\end{eqnarray*}
By classical variational arguments and the fact that $X$ is compactly embedded into $L^2(\bbr^d)$, we can see that $\tau_1 > -\infty$ is attained.
Since the linear operator
$$
L_{\infty} + \omega_{\infty} - \tau_1 + 1
= -\Delta+|x|^2-3u_\infty^2 - \tau_1 + 1
$$
is strictly positive in $X_{\rm rad}$, the linear equation
	\begin{eqnarray}\label{eq4458}
	-\Delta\psi+(|x|^2-3u_\infty^2-\tau_1 + 1)\psi=\varphi\quad\text{in }X_{\rm rad},
	\end{eqnarray}
	is unique solvable for every $\varphi \in X_{\rm rad}$.  Let $\{\varphi_n\}_{n \in \mathbb{N}}$ be bounded in $X_{\rm rad}$, then it follows by the compactness of the embedding from $X$ to $L^2(\bbr^d)$ that $\varphi_n\to\varphi_*$ as $n \to \infty$ strongly in $L^2(\bbr^d)$.  Since the equation \eqref{eq4458} is linear, we may assume that $\varphi_*=0$.  By the positivity of $L_{\infty} + \omega_{\infty} - \tau_1 + 1$ in $X_{\rm rad}$, $\{\psi_n\}_{n \in \mathbb{N}}$ is bounded in $X_{\rm rad}$.  Since $\varphi_n\to0$ strongly in $L^2(\bbr^d)$ as $n \to \infty$, then $\psi_n\to0$ as $n \to \infty$ strongly in $X_{\rm rad}$.  Therefore, $L_{\infty} + \omega_{\infty} - \tau_1 + 1$ has a compact resolvent in $X_{\rm rad}$, and so does $L_{\infty}$.
\end{proof}

\begin{remark}
	The mapping $d \mapsto \frac{12(d-3)}{(d-2)^2}$ is monotonically decreasing for $d \geq 5$. Since $\frac{12(d-3)}{(d-2)^2} < 1$ for $d \geq 13$,
	we have $\tau_1 > 0$ for $d \geq 13$. However, $\tau_1 < 0$ for $5 \leq d \leq 12$.
\end{remark}

Let $\mathfrak{m}(u_{\infty})$ be the Morse
index of $u_{\infty}$ in $X_{\rm rad}$ according to Definition \ref{def-Morse}.
It is well-defined for $d \geq 5$ because $L_{\infty}$ has a purely discrete spectrum of isolated (simple) eigenvalues by Lemma \ref{lem-compact}.

\subsection{Morse index in the oscillatory case}

The following lemma shows that the Morse index of $u_{\infty}$ is infinite for $5 \leq d \leq 12$, for which $\omega_b$ oscillates near $\omega_{\infty}$
as $b \to \infty$.

\begin{lemma}\label{lem0008}
For $5\leq d\leq12$, we have $\mathfrak{m}(u_\infty)=\infty$.
\end{lemma}

\begin{proof}
We consider the following two cases:
\begin{enumerate}
\item[$(1)$]\quad There exists $b_n\to+\infty$ as $n\to\infty$ such that $\omega_{b_n}-\omega_\infty>0$.
\item[$(2)$]\quad $\omega_b\leq\omega_\infty$ for $b>0$ sufficiently large.
\end{enumerate}

{\bf Case~$(1)$.}  By using equations (5.4), (6.30), and (6.47) from \cite{BFPS21}, we obtain
\begin{eqnarray}\label{eqnew0056}
u_\infty(r) = \frac{\sqrt{d-3}}{r}- \frac{\omega_\infty\sqrt{d-3}}{4d-10}r+\mathcal{O}(r^{3})
\end{eqnarray}
and
\begin{eqnarray}\label{eqnew0052}
u_{b_n}(r) = \frac{\sqrt{d-3}}{r}+ C(\omega_{b_n}-\omega)r^{-\beta-1}\sin(\alpha\log r+\delta)+\mathcal{O}(b_n^{-2(1-a)}+\ve^2),
\end{eqnarray}
for $r = \mathcal{O}(b_n^{a-1})$,
where $|\omega_{b_n}-\omega_\infty| = \mathcal{O}(\ve b_n^{-\beta(1-a)})$,
$C\in \mathbb{R}$, $\delta \in \mathbb{R}$, $\ve>0$ is sufficiently small, $a\in(0, 1)$ and
$$
\beta=\frac{d-4}{2}, \qquad \alpha=\frac{\sqrt{-d^2+16d-40}}{2}.
$$
Let $\varphi_{b_n} := u_\infty-u_{b_n}$.
Since $u_{b_n}(0)=b_n$ and $\omega_{b_n}-\omega_\infty>0$,
it follows from \eqref{eqnew0056} and \eqref{eqnew0052} that there exists $r_{b_n}\to0$ such that $\varphi_{b_n}(r) > 0$ for $r \in (0, r_{b_n})$ and $\varphi_{b_n}(r_{b_n})=0$.  It follows from \eqref{eqnew0001} that
$\varphi_{b_n}$ satisfies for $r \in (0,r_{b_n})$:
\begin{align}
-\Delta\varphi_{b_n}+|x|^2\varphi_{b_n} &= (\omega_\infty+u_\infty^2)\varphi_{b_n}- (u_{b_n}^2-u_\infty^2+\omega_{b_n}-\omega_\infty) u_{b_n}\notag\\
&=(\omega_\infty+3u_\infty^2)\varphi_{b_n}-(2u_\infty^2-u_\infty u_{b_n}-u_{b_n}^2)\varphi_{b_n}-(\omega_{b_n}-\omega_\infty)u_{b_n}\notag\\
&<(\omega_\infty+3u_\infty^2)\varphi_{b_n}. \label{eqnew6666}
\end{align}
Let
\begin{eqnarray*}
\widetilde{\varphi}_{b_n}=\left\{ \begin{array}{ll} \varphi_{b_n}, &\quad 0<r<r_{b_n},\\
0, &\quad r\geq r_{b_n}.
\end{array} \right.
\end{eqnarray*}
Then by multiplying \eqref{eqnew6666} with $\widetilde{\varphi}_{b_n}$ on both sides and integrating by parts, we have
\begin{eqnarray*}
\int_{\bbr^d} \left( |\nabla \widetilde{\varphi}_{b_n}|^2+|x|^2|\widetilde{\varphi}_{b_n}|^2 \right) dx < \int_{\bbr^d}(\omega_\infty+3u_\infty^2)|\widetilde{\varphi}_{b_n}|^2dx.
\end{eqnarray*}
Since $r_{b_n}\to0$ as $n\to\infty$, $\{\widetilde{\varphi}_{b_n}\}$ is linearly independent up to a subsequence.  Hence, $\mathfrak{m}(u_{\infty}) = \infty$. \\

{\bf Case~$(2)$.}  We follow the idea in \cite{GW11}.  Let $W_b=u_b(e^t)$ and $W_\infty=u_\infty(e^t)$, then $Z_b=\frac{W_b}{W_\infty}$ satisfies
\begin{eqnarray}
\label{eqnew0057}
Z_b''+(d-2+\frac{2W_\infty'}{W_\infty})Z_b'+e^{2t}Z_b (\omega_b-\omega_\infty  + W_\infty^2 (Z_b^2-1)) = 0.
\end{eqnarray}
It follows from the convergence $u_b \to u_{\infty}$ in $\Sigma$ 
by \cite[Theorem 1.2]{SKW13}
that $Z_{b_n}(t)\to1$ as $n\to+\infty$ for every fixed $t$.  Moreover, by classical elliptic regularity, we also have $u_b \to u_{\infty}$ in $C^{1,\alpha}_{loc}(\bbr^d\backslash\{0\})$ as $b\to+\infty$.  We claim that there exists $b_n\to+\infty$ as $n\to\infty$ such that $1-Z_{b_n}(t)$ has at least $n$ zeros, say $t_{n,n}<\cdots<t_{2,n}<t_{1,n}$, such that $t_{n,n}\to 0$ as $n\to+\infty$. In other words, we claim that $Z_{b_n}$ is oscillatory around $1$ as $n\to\infty$ on $(-\infty,0)$, in agreement with (\ref{eqnew0052}).

Suppose the contrary.  Then for every sequence $\{b_n\}$ satisfying $b_n\to+\infty$ as $n\to\infty$, there exists $N>0$, independent of $n$, such that $1-Z_{b_n}(t)$ has at most $N$ zeros for all $n$. Since $Z_b(t) = \mathcal{O}(e^{t})$ as $t \to -\infty$ by \cite[(3.9)]{BFPS21} for every $b>0$ there exists $t_0>0$, independent of $n$, such that $0 < Z_{b_n}(t)<1$ for all $t<t_0$ and $n$. If $V_{b_n}=1-Z_{b_n}$, then $0<V_{b_n}(t)<1$ for $t<t_0$.  Moreover, by \eqref{eqnew0057}, $V_{b_n}$ satisifes
\begin{eqnarray*}
V_{b_n}''+(d-2+\frac{2W_\infty'}{W_\infty})V_{b_n}'-e^{2t}Z_{b_n}
(\omega_{b_n}-\omega_\infty - W_\infty^2 (Z_{b_n}+1)V_{b_n})=0.
\end{eqnarray*}
Since $u_{b_n} \to u_{\infty}$ in $C^{1,\alpha}_{\rm loc}(\bbr^d\backslash\{0\})$ as $n\to \infty$, we know that $Z_{b_n}(t)\to1$ as $n\to \infty$ uniformly in every compact set of the interval $(-\infty, t_0]$.  Note that we also have $e^{2t} W_{\infty}^2 \to (d-3)$ as $t\to-\infty$, thus, there exists $t_0'<t_0$ which is independent of $n$, such that $e^{2t} W_{\infty}^2=(d-3)+o(1)$ for $t<t_0'$ where $o(1)\to0$ as $t_0'\to-\infty$.  Thus, without loss of generality, we may assume that $e^{2t}Z_{b_n}W_\infty^2(Z_{b_n}+1)=2(d-3)+o(1)$ uniformly in every compact set of the interval $(-\infty, t_0']$, where $o(1)$ could be arbitrary small if necessary by taking $t_0'$ sufficiently close to $-\infty$ and $n$ sufficiently large.  Note that by \eqref{eqnew0056},
$$
\frac{2W_\infty'(t)}{W_\infty(t)} \to -2, \quad \mbox{\rm as} \;\; t\to-\infty.
$$
Since $\omega_{b_n} \leq \omega_\infty$ is obtained from the assumption, in this case, we can write the equation of $V_{b_n}$ as follows:
\begin{eqnarray*}
V_{b_n}''+(d-4+o(1))V_{b_n}'+(2(d-3)+o(1))V_{b_n}\leq0
\end{eqnarray*}
in every compact set of the interval $(-\infty, t_0']$ by taking $t_0'$ sufficiently close to $-\infty$ if necessary.
Since $5\leq d\leq12$, the fundamental solution of the linear equation,
\begin{eqnarray*}
\phi''+(d-4)\phi'+ 2(d-3)\phi=0,
\end{eqnarray*}
is given by $\phi = C e^{-\beta t}\sin(\alpha t+\delta)$ for some
$C \in \mathbb{R}$ and $\delta \in \mathbb{R}$.  By the Sturm--Liouville theorem, $V_{b_n}$ must have zeros in a sufficiently large compact set of the interval $(-\infty, t_0']$.  It contradicts the assumption that $V_{b_n}(t) >0$ for all $t<t_0'$.  Thus, there exists $b_n\to+\infty$ as $n\to\infty$ such that $1-Z_{b_n}(t)$ has at least $n$ zeros for $t \ll -1$.  We assume the zeros of $Z_{b_n}$ by $0<a_{1,n}<a_{2,n}<\cdots<a_{k_n,n}$ with $k_n\geq n$.  For the sake of simplicity, we also denote $a_{0,n}=0$.  Then we can define
\begin{align*}
\widehat{\varphi}_{n,j}=\left\{\begin{array}{ll} 0, & \quad 0<r\leq a_{j-1,n},\\
u_\infty-u_{b_n}, & \quad a_{j-1,n}<r<a_{j,n},\\
0, & \quad r\geq a_{j,n}\end{array}\right.
\end{align*}
and by the convexity of $t^3$, we have
\begin{eqnarray*}
\int_{\bbr^d} \left( |\nabla \widehat{\varphi}_{n,j}|^2+|x|^2|\widehat{\varphi}_{n,j}|^2 \right) dx < \int_{\bbr^d}(\omega_\infty+3u_\infty^2)|\widehat{\varphi}_{n,j}|^2dx.
\end{eqnarray*}
It follows from $k_n\to\infty$ as $n\to\infty$ that $\mathfrak{m}(u_{\infty}) = \infty$.
\end{proof}

\subsection{Morse index in the monotone case}

By Theorems 1.1 and 1.2 in \cite{PS22}, the Morse index of $u_{\infty}$ is finite for $d \geq 13$ for which $\omega_b$ converges to $\omega_{\infty}$ monotonically as $b \to \infty$. Here we will give a more precise estimates on $\mathfrak{m}(u_{\infty})$ for $d\geq13$.

Let us consider the confluent hypergeometric function, which is also called Kummer's function, given by
\begin{eqnarray*}
M(a; b; x)=\sum_{n=0}^{\infty}\frac{(a)_n}{(b)_n}\frac{x^n}{n!},
\end{eqnarray*}
where $(\alpha)_n=\alpha(\alpha+1)\cdots(\alpha+n-1)$
are Pochhammer symbols.  It is well known (cf. \cite{V16}) that $M(a; b; x)$ is a solution of the confluent hypergeometric differential equation, which is also called the Kummer equation:
\begin{eqnarray*}
x\frac{d^2u}{dx^2}+(b-x)\frac{du}{dx}+au=0.
\end{eqnarray*}
Let
\begin{eqnarray*}
W_{a,l}(r)=r^{l}e^{-\frac{r^2}{2}}M(a; l+\frac{d}{2}; r^2),
\end{eqnarray*}
then it can be directly verified that $W$ satisifies
\begin{eqnarray*}
-W_{a,l}''-\frac{d-1}{r}W_{a,l}'+\frac{l(l+d-2)}{r^2}W_{a,l}+r^2W_{a,l}=(d-4a+2l)W_{a,l}.
\end{eqnarray*}
Let
\begin{eqnarray}\label{eqnew0012}
l_{\pm}=\frac{2-d\pm\sqrt{d^2-16d+40}}{2}
\end{eqnarray}
then $W_{a,l_\pm}$ satisfies
\begin{eqnarray}\label{eqnew0008}
-\Delta W_{a,l_\pm}+|x|^2W_{a,l_\pm}-\frac{3(d-3)}{|x|^2}W_{a,l_\pm}=(d-4a+2l_{\pm})W_{a,l_\pm}.
\end{eqnarray}
\begin{remark}
It is easy to see that $b=l_{\pm}+\frac{d}{2}\not=0,-1,-2,\cdots$.  Otherwise, we have
\begin{eqnarray*}
\frac{d^2-16d+40}{4}-p^2=0
\end{eqnarray*}
for some $p\in\mathbb{Z}$, which implies
\begin{eqnarray*}
d=2(4\pm\sqrt{p^2+6})\in\bbn.
\end{eqnarray*}
It follows that $(\frac{q}{2})^2-p^2=6$ for some $q\in\mathbb{Z}$.  Thus, either $\frac{q}{2}-p=2k$ or $\frac{q}{2}+p=2k$ for some $k\in\bbn$, which implies $4k^2\pm4pk=6$.  It is impossible since $2k^2$ is even but $3\pm 2kp$ is odd.
\end{remark}

If $a\not=0,-1,-2,\cdots$ then
\begin{eqnarray*}
M(a; l_{\pm}+\frac{d}{2}; r^2)\sim\sum_{n=0}^{\infty}n^{l_{\pm}+\frac{d}{2}-a}\frac{r^{2n}}{n!}\gtrsim\sum_{n=0}^{\infty}\frac{(\frac{2}{3}r^2)^{n}}{n!}=e^{\frac{2}{3}r^2}.
\end{eqnarray*}
If $-a\in\bbn$, then $M(-n; l_{\pm}+\frac{d}{2}; r^2)=P_n(r^2)$ is a polynomial of order $2n$.  Therefore, $W_{a,l_\pm}\in L^2(\bbr^d)$ if and only if $-a\in\bbn$.  On the other hand, if $W_{a,l_\pm}\in L^2(\bbr^d)$ is a eigenfunction of the operator $-\Delta+|x|^2-\frac{3(d-3)}{|x|^2}$ in $L^2(\bbr^d)$, then $W_{a,l_\pm}\in L^{\frac{2d}{d-2}}(\bbr^d)$ by the Hardy inequality for $d\geq13$.  However, as $r\to0$,
\begin{eqnarray*}
|r^{l_{-}}e^{-\frac{r^2}{2}}M(-n; l_{-}+\frac{d}{2}; r^2)|^{2^*}\sim r^{2^*l_{-}}\sim r^{-d-\frac{d\sqrt{d^2-16d+40}}{d-2}}>r^{-d}.
\end{eqnarray*}
Thus, by \eqref{eqnew0008},
\begin{eqnarray*}
W_{-n,l_+}=r^{l_{+}}e^{-\frac{r^2}{2}}M(-n; l_{+}+\frac{d}{2}; r^2)
\end{eqnarray*}
is the only eigenfunctions of the operator $-\Delta+|x|^2-\frac{3(d-3)}{|x|^2}$ in $X_{rad}$ with eigenvalues $(d+4n+2l_{+})$, for all $n\in\bbn$.  By \eqref{eqnew0012}, the third eigenvalue $\sigma_3$ is given by
\begin{eqnarray}\label{eqnew0004}
\sigma_3=10+\sqrt{d^2-16d+40}
\end{eqnarray}
and
the fourth eigenvalue $\sigma_4$ is given by
\begin{eqnarray}\label{eqnew0034}
\sigma_4=14+\sqrt{d^2-16d+40}.
\end{eqnarray}
The following lemma gives the estimate on $\mathfrak{m}(u_{\infty})$ for $d \geq 13$.

\begin{lemma}\label{lem0009}
For $d\geq13$, we have
\begin{equation*}
\mathfrak{m}(u_\infty)=\left\{\begin{array}{ll}
1\text{ or }2, &\quad 13\leq d\leq 15,\\
1, &\quad d\geq16.\end{array} \right.
\end{equation*}
\end{lemma}

\begin{proof}
	{\bf Case $d \geq 16$.} Since $F(r) := r u_{\infty}(r)$ is monotonically decreasing (cf. \cite{BFPS21,SKW13}), we have $F(r) < F(0) = \sqrt{d-3}$, which implies that $u_{\infty}(r) < \frac{\sqrt{d-3}}{r}$ for every $r > 0$.  Note that $\omega_{\infty} \in (d-4,d)$ by \cite[Theorem~1.2]{BFPS21}. Then by $\sigma_3 > d$ for $d \geq 16$, as is clear from \eqref{eqnew0004},
we have
\begin{eqnarray}\label{eqnew0006}
\omega_\infty+3u_\infty^2<\sigma_3+3\frac{d-3}{r^2}\quad\text{in }\bbr^d\text{ for }d\geq16.
\end{eqnarray}
Since $L_{\infty}$ has a compact resolvent in $X_{\rm rad}$ by Lemma \ref{lem-compact}, the spectrum of $-\Delta+|x|^2-3u_\infty^2$ in $X_{\rm rad}$ consists of isolated (simple) eigenvalues $\{ \tau_j \}_{j \in \mathbb{N}}$ such that $\tau_j \to \infty$ as $j \to \infty$. For each simple eigenvalue $\tau_j$, there exists a unique eigenfunction $\phi_j\in X_{\rm rad}$ (up to scalar multiplication) which satisfies
\begin{eqnarray*}
-\Delta \phi_j+|x|^2\phi_j-3u_\infty^2\phi_j=\tau_j\phi_j\quad\text{in }\bbr^d.
\end{eqnarray*}
Moreover, $\phi_j$ has exact $j-1$ zeros.  Since
\begin{eqnarray*}
\int_{\bbr^d} \left( |\nabla u_\infty|^2+|x|^2u_\infty^2 \right) dx = \int_{\bbr^d} \left( \omega_\infty u_\infty^2+u_\infty^4 \right) dx < \int_{\bbr^d} \left( \omega_\infty u_\infty^2+3u_\infty^4 \right) dx,
\end{eqnarray*}
we have $\mathfrak{m}(u_{\infty}) \geq 1$ so that $\tau_1 <\omega_\infty$.  Suppose that $\tau_2\leq\omega_\infty$, then it follows from \eqref{eqnew0006}
that
\begin{eqnarray}\label{eqnew0007}
\tau_2+3u_\infty^2<\sigma_3+3\frac{d-3}{r^2}\quad\text{in }\bbr^d\text{ for }d\geq16.
\end{eqnarray}
Recall that $\phi_2$ has exact one zero on $(0,\infty)$ so that we can define
\begin{eqnarray*}
\phi_{2,f}=\left\{\begin{array}{ll} \phi_2, &\quad 0\leq r<r_u,\\
0, &\quad r\geq r_u,\end{array} \right.
\quad\text{and}\quad
\phi_{2,l}=\left\{\begin{array}{ll} 0, & \quad 0\leq r<r_u,\\
\phi_2, &\quad r\geq r_u,\end{array} \right.
\end{eqnarray*}
where $r_u$ is the unique zero of $\phi_2$.  Then by \eqref{eqnew0007}, we have
\begin{eqnarray*}
\int_{\bbr^d} \left( |\nabla \phi_1|^2+|x|^2\phi_1^2 \right) dx = \int_{\bbr^d}(\tau_1+3u_\infty^2)\phi_1^2dx<\int_{\bbr^d}(\sigma_3+3\frac{d-3}{r^2})\phi_1^2dx,
\end{eqnarray*}
\begin{eqnarray*}
\int_{\bbr^d} \left( |\nabla \phi_{2,f}|^2+|x|^2\phi_{2,f}^2 \right) dx=\int_{\bbr^d}(\tau_2+3u_\infty^2)\phi_{2,f}^2dx<\int_{\bbr^d}(\sigma_3+3\frac{d-3}{r^2})\phi_{2,f}^2dx
\end{eqnarray*}
and
\begin{eqnarray*}
\int_{\bbr^d} \left( |\nabla \phi_{2,l}|^2+|x|^2\phi_{2,l}^2 \right) dx=\int_{\bbr^d}(\tau_2+3u_\infty^2)\phi_{2,l}^2dx<\int_{\bbr^d}(\sigma_3+3\frac{d-3}{r^2})\phi_{2,l}^2dx.
\end{eqnarray*}
Since $\phi_1$ is sign-constant and $\phi_{2,f}$ and $\phi_{2,l}$ share the unique zero at $r_u$, the functions $\phi_1$, $\phi_{2,f}$ and $\phi_{2,l}$ are linearly independent.  Indeed, if there exists $c_1$, $c_{2,f}$ and $c_{2,l}$ such that
$$
c_1\phi_1+c_{2,f}\phi_{2,f}+c_{2,l}\phi_{2,l}\equiv 0 \quad \mbox{\rm in} \;\; \bbr^d,
$$
then by $\phi_{2,f}(r_u)=\phi_{2,l}(r_u)=0$, we have $c_1=0$. On the other hand, since $\phi_{2,f}\phi_{2,l}\equiv0$, then we also have $c_{2,f}=c_{2,l}=0$, which implies $\phi_1$, $\phi_{2,f}$ and $\phi_{2,l}$ are linearly independent.  However, $\sigma_3$ is the third eigenvalue of the operator $-\Delta+|x|^2-3\frac{d-3}{r^2}$ in $X_{\rm rad}$, thus,
$\mathfrak{m}(W_{-2,l_+}) = 2$, which is a contradiction.  Therefore, $\tau_2 > \omega_\infty$ for $d\geq16$, which implies $\mathfrak{m}(u_{\infty}) = 1$. \\

{\bf Case $13\leq d\leq15$.} We use the same idea to show that $1 \leq \mathfrak{m}(u_{\infty}) \leq 2$.  Indeed,
since $\sigma_4 > \omega_{\infty}$ for $13 \leq d \leq 15$, as follows from \eqref{eqnew0034}, we have
\begin{eqnarray}\label{eqnew0036}
\omega_\infty+3u_\infty^2<\sigma_4+3\frac{d-3}{r^2}\quad\text{in }\bbr^d\text{ for }13\leq d\leq15.
\end{eqnarray}
If $\tau_3\leq\omega_\infty$, then by \eqref{eqnew0036},
\begin{eqnarray}\label{eqnew0037}
\tau_3+3u_\infty^2<\sigma_4+3\frac{d-3}{r^2}\quad\text{in }\bbr^d\text{ for }13\leq d\leq15.
\end{eqnarray}
The third eigenfunction $\phi_3$, corresponding to $\tau_3$, has exact two zeros $\widetilde{r}_f<\widetilde{r}_l$.
Moreover, by the Stum-Liouville theorem, it is well known that $\widetilde{r}_f<r_u<\widetilde{r}_l$.  Let
\begin{eqnarray*}
\phi_{3,f}=\left\{\begin{array}{ll} \phi_3, &\quad 0\leq r<\widetilde{r}_f,\\
0, &\quad r\geq \widetilde{r}_f,\end{array}\right.
\quad
\quad
\phi_{3,l} =\left\{\begin{array}{ll} 0, &\quad 0\leq r<\widetilde{r}_l,\\
\phi_3, & \quad r\geq \widetilde{r}_l,\end{array} \right.
\end{eqnarray*}
and
\begin{eqnarray*}
\phi_{3,m}=\left\{\begin{array}{ll} 0, &\quad 0\leq r<\widetilde{r}_f,\\
	\phi_3, & \quad r_f\leq r< \widetilde{r}_l,\\
	0, & \quad r\geq \widetilde{r}_l.\end{array}\right.
\end{eqnarray*}
Then by similar arguments as used above, we can show from \eqref{eqnew0037}
that $\mathfrak{m}(W_{-3,l_+}) \geq 6$, which contradicts the fact that  $\mathfrak{m}(W_{-3,l_+}) = 3$.  Thus, we must have $\tau_3 > \omega_\infty$ for $13\leq d\leq15$, which implies that  $\mathfrak{m}(u_{\infty}) \leq 2$.
\end{proof}

\begin{remark}\label{rmk0001}
As a by-product, the proof of Lemma~\ref{lem0009} shows that $\tau_1<\omega_\infty<\tau_2$ for $d\geq16$.  Therefore, the homogeneous
equation $L_{\infty} Z = 0$ has only  trivial solutions in $X_{\rm rad}$ for $d\geq16$.  This implies that $u_\infty$ is nondegenerate in $X_{\rm rad}$ for $d\geq16$ in the following sense. The radial $Z$ satisfies $Z = \mathcal{O}( r^{\omega_\infty-d}e^{-\frac{r^2}{2}})$ as $r\to+\infty$ and there exists $L_-\not=0$ such that
\begin{eqnarray*}
Z=L_-r^{l_-}+O(r^{l_+}, r^{l_-+2})\quad\text{as }r\to0.
\end{eqnarray*}
This argument verifies the non-degeneracy Assumption 2.2 in \cite{PS22} for $d\geq16$. It is not clear if this assumption can be verified for $13 \leq d \leq 15$.
\end{remark}

\vskip0.12in

The proof of Theorem~\ref{thm0002}  follows immediately from Lemmas~\ref{lem0008} and \ref{lem0009}.
\hfill$\Box$

\vskip0.2in

{\bf Acknowledgements.} The research of J. Wei and D. E. Pelinovsky is
partially supported by the NSERC Discovery grants. 
The research of Y. Wu is supported by the NSFC grants (No. 11971339, 12171470).

\end{document}